\documentclass[11pt]{amsart}

\usepackage{a4wide}
\usepackage{wasysym}
\usepackage[usenames,dvipsnames,svgnames,table]{xcolor}
\usepackage[colorlinks=true, pdfstartview=FitV,
linkcolor=ForestGreen,citecolor=ForestGreen, urlcolor=blue]{hyperref}

\usepackage{amsfonts,amssymb,amsmath,amsthm}
\usepackage{latexsym}
\usepackage{color}
\usepackage{amscd}
\usepackage{graphicx}
\usepackage{enumitem} 
\usepackage[margin=1in]{geometry}
\usepackage{esint} 
\usepackage{epsfig}

\definecolor{labelkey}{rgb}{0,0,1}

\numberwithin{equation}{section}

\newtheorem{theorem}{Theorem}
\newtheorem{proposition}[theorem]{Proposition}
\newtheorem{lemma}[theorem]{Lemma}

\theoremstyle{definition}

\theoremstyle{remark}
\newtheorem{remark}[theorem]{Remark}

\newcommand\N{{\mathbb N}}
\newcommand\R{{\mathbb R}}

\newcommand\Q{{\mathbb Q}}

\newcommand{\ds}{\displaystyle}

\DeclareMathOperator{\supp}{supp}

\DeclareMathOperator{\osc}{osc}

\newcommand{\cQ}{\mathcal Q}
\newcommand{\cR}{\mathcal R}

\def\eps{{\varepsilon}}
\def\un{\mathbf{1}}
\def\Qel{Q^{\mathit{el}}}
\def\Qint{Q_{\mathrm{int}}}
\def\Q12{Q_{\mathrm{mid}}}
\def\Qext{Q_{\mathrm{ext}}}

\newcommand{\dd}{{\, \mathrm d}}
\newcommand{\n}[1]{{\left\| #1 \right\|}}

\def\signfg{\bigskip \begin{center} {\sc Fran\c{c}ois Golse \par\vspace{3mm}
      \'Ecole polytechnique, Centre de math\'ematiques Laurent Schwartz  \par
      91128 Palaiseau Cedex, France
\par\vspace{3mm} e-mail:} \tt{francois.golse@polytechnique.edu} \end{center}}

\def\signci{\bigskip \begin{center} {\sc Cyril Imbert\par\vspace{3mm}
      CNRS \& \'ENS, D\'epartement de math\'ematiques et applications \par
      UMR 8553, \'Ecole normale sup\'erieure (Paris)\par
 45 rue d'Ulm, F-75230 Paris cedex 5, 
      France\par\vspace{3mm} e-mail:}
    \tt{Cyril.Imbert@ens.fr} \end{center}}

\def\signcm{\bigskip \begin{center} {\sc Cl\'ement
      Mouhot\par\vspace{3mm}
      University of Cambridge\par
      DPMMS, Centre for Mathematical Sciences\par
      Wilberforce road, Cambridge CB3 0WA, UK
      \par\vspace{3mm} e-mail:}
    \tt{C.Mouhot@dpmms.cam.ac.uk} \end{center}}

\def\signav{\bigskip \begin{center} {\sc Alexis F. Vasseur\par\vspace{3mm}
      Department of Mathematics, University of Texas at Austin \par
      1 University Station - C1200, Austin,
      Texas, TX 78712-0257, USA\par\vspace{3mm} e-mail:}
    \tt{vasseur@math.utexas.edu} \end{center}}

\begin{document}

\title[Harnack inequality for kinetic Fokker-Planck equations]{Harnack
  inequality for kinetic Fokker-Planck equations with rough
  coefficients and application to the Landau equation}

\author{F. Golse, C. Imbert, C. Mouhot and A. F. Vasseur} \thanks{The
  authors would like to thank Luis Silvestre for fruitful comments
  during the preparation of this article. The work of A. F. Vasseur
  was partially supported by the NSF Grant DMS 1209420, and by a
  visiting professorship at \'Ecole polytechnique. The work of
  C. Mouhot was partially supported by the ERC Grant MATKIT, and by a
  visiting professorship at Universit\'e Paris-Est Cr\'eteil.}

\date{\today}


\begin{abstract}
  We extend the De~Giorgi--Nash--Moser theory to a class of kinetic
  Fokker-Planck equations and deduce new results on the Landau-Coulomb
  equation. More precisely, we first study the H\"older regularity and
  establish a Harnack inequality for solutions to a general linear
  equation of Fokker-Planck type whose coefficients are merely
  measurable and essentially bounded, i.e.  assuming no regularity on
  the coefficients in order to later derive results for non-linear
  problems.  This general equation has the formal structure of the
  hypoelliptic equations ``of type II'', sometimes also called
  ultraparabolic equations of Kolmogorov type, but with rough
  coefficients: it combines a first-order skew-symmetric operator with
  a second-order elliptic operator involving derivatives along only
  part of the coordinates and with rough coefficients. These general
  results are then applied to the non-negative essentially bounded
  weak solutions of the Landau equation with inverse-power law
  $\gamma \in [-d,1]$ whose mass, energy and entropy density are
  bounded and mass is bounded away from $0$, and we deduce the
  H\"older regularity of these solutions.
\end{abstract}

\maketitle

\tableofcontents

\section{Introduction}
\label{sec:introduction}

\subsection{The Landau equation}

We consider the Landau equation 
\begin{equation}
\label{eq:landau0}
\partial_t f + v \cdot \nabla_x f = \nabla_v \cdot (A[f] \nabla_v f + B[f] f)
\end{equation}
where 
\begin{equation*}
\begin{cases}
\ds A[f] (v) = a_{d,\gamma} \int_{\R^d} \left( I - \frac{w}{|w|}\otimes \frac{w}{|w|}\right) |w|^{\gamma+2} f(v-w) \dd w,\\[3mm]
\ds B[f](v) = b_{d,\gamma} \int_{\R^d} |w|^{\gamma}w f(v-w) \dd w
\end{cases}
\end{equation*}
with $\gamma \in [-d,0]$ and $a_{d,\gamma}>0$. We note that the main
physical case is that of Coulomb interactions when $\gamma=-d$ and
$d=3$ (giving rise to the \emph{Landau-Coulomb equation} in plasma
physics); the other cases are \emph{hard potentials}
$\gamma \in (0,1]$ (not covered here\footnote{Our method would apply
  as well in this case with no changes, we did not include it only
  because it requires additional the condition
  $\sup_x \int_v f(t,x,v)|v|^{2+\gamma} \dd v < \infty$ on the
  solution and we wanted a clean statement.}), \emph{Maxwellian
  molecules} $\gamma=0$, and \emph{soft potentials}
$\gamma \in [-d,0)$. It can be rewritten as follows
\begin{equation}
\label{eq:landau}
\partial_t f + v \cdot \nabla_x f = \nabla_v \cdot \left(A[f] \nabla_v f \right) + B[f] \nabla_v f + c[f] f
\end{equation}
where 
\[ 
c[f] (v) = 
\begin{cases} 
  \ds c_{d,\gamma} \int_{\R^d} |w|^\gamma f(v-w) \dd w & \text{ if } \gamma > -d , \\[2mm]
  \ds c_{d,\gamma} f & \text{ if } \gamma = -d .
\end{cases}
\]
We assume that the mass, energy and entropy density of the weak
solution $f$ satisfy the following control at a given space-time point
$(x,t)$: 
\begin{equation}
\label{e:meh}
\mathfrak{C}(x,t) \quad 
\left\{
\begin{aligned}
M_1 \le M (x,t) &= \int_{\R^d} f(x,v,t) \dd v \le M_0 &&  \text{
  (local mass), } \\
E(x,t) &= \frac12 \int_{\R^d} f(x,v,t) |v|^2 \dd v \le E_0 && \text{
  (local energy), } \\
H(x,t) &= \int_{\R^d} f(x,v,t) \ln f (x,v,t) \dd v \le H_0 &&  \text{
  (local entropy). }
\end{aligned}
\right.
\end{equation}

The \emph{weak solutions} to equation~\eqref{eq:landau0} on
$U_x \times U_v \times I$, $U_x \subset \R^d$ open, $U_v \subset \R^d$
open, $I = [a,b]$ with $-\infty<a<b \le +\infty$, are defined as
functions
$f \in L^\infty_t(I,L^2_{x,v}(U_x \times U_v))) \cap L^2_{x,t}(U_x
\times I, H^1_v(U_v))$ such that
$\partial_t f + v\cdot \nabla_x f \in L^2_{x,t}(U_x \times I,
H^{-1}_v(U_v))$, $f$ satisfies estimates~\eqref{e:meh} and satisfies
the equation in the sense of distributions\footnote{Observe that the
  coefficients $A[f]$ and $B[f]$ are controlled under
  assumption~\eqref{e:meh}, thanks to Lemmas~\ref{lem:upperbound}
  and~\ref{lem:lowerbound}.}.

\begin{theorem}[H\"older continuity for the Landau equation]\label{thm:holder-landau}
  Assume $\gamma \in [-d,0]$.  Let $f$ be an essentially bounded weak
  solution of \eqref{eq:landau} in $B_1 \times B_1 \times (-1,0]$.
  Assume that $\mathfrak C(x,t)$ (equation~\eqref{e:meh}) holds true
  for all $B_1 \times (-1,0]$.  Then $f$ is $\alpha$-H\"older
  continuous with respect to
  $(x,v,t)\in B_{\frac12}\times B_{\frac12}\times (-\frac12,0]$ and
  \[ \| f \|_{C^\alpha\left(B_{1/2} \times B_{1/2} \times
    (-1/2,0]\right)} \le C \left(\|f \|_{L^2(B_1 \times B_1 \times
      (-1,0])} + \|f\|^{1+\frac{|\gamma|}d}_{L^\infty(B_1 \times B_1
      \times (-1,0])} \right)\]
  for some $\alpha$ and $C$ depending on dimension, $M_1$, $M_0$,
  $E_0$ and $H_0$.
\end{theorem}

\begin{remark}
  After this work was completed, we heard from a nice recent preprint
  of Cameron, Silvestre and Snelson \cite{CSS} that establishes
  \emph{a priori} upper bounds for solutions to the spatially
  inhomogeneous Landau equation in the case of moderately soft
  potentials ($\gamma \in [-2,0]$), with arbitrary initial data, under
  the assumption~\eqref{e:meh}. When $\gamma \in [-2,0]$, it thus
  allows us to remove the $L^\infty$ assumption on the weak solution
  in Theorem~\ref{thm:holder-landau}.
\end{remark}
Under the assumptions of Theorems~\ref{thm:holder-landau}, it is known
\cite{dv,luislandau} that the diffusion matrix $A[f]$ is uniformly
elliptic and $B[f]$ and $c[f]$ are essentially bounded for bounded
velocities (see Lemmas~\ref{lem:upperbound} and \ref{lem:lowerbound}
in Appendix).  In particular, the assumption~\eqref{eq:ellipticity}
given below, and under which our main results
(Theorems~\ref{thm:holder} and \ref{thm:harnack}) hold true, is
satisfied.

\subsection{The question studied and its history}

We are also motivated by the study of the following nonlinear kinetic
Fokker-Planck equation
\begin{equation}
  \label{eq:1}
  \partial_t f + v \cdot \nabla_x f = \rho [f] \, \nabla_v \cdot \left( \nabla_v
    f + v f \right), \quad t \ge 0, \ x \in \R^d, \ v  \in \R^d, 
\end{equation}
(with or without periodicity conditions with respect to the space
variable) where $d \in \N^*$, $f = f(x,v,t) \ge 0$ and $\rho[f] :=
\int_{\R^d} f(x,v,t) \dd v$. The construction of global smooth
solutions for such a problem is one motivation of the present paper.

The linear kinetic Fokker-Planck equation
$\partial_t f + v \cdot \nabla_x f = \nabla_v \cdot \left( \nabla_v f
  + v f \right)$
is sometimes called the Kolmogorov-Fokker-Planck equation, as it was
studied by Kolmogorov in the seminal paper \cite{kolm}. In this note,
Kolmogorov explicitely calculated the fundamental solution and deduced
regularisation in both variables $x$ and $v$, even though the operator
$\nabla_v \cdot (\nabla_v + v) - v \cdot \nabla_x$ shows ellipticity
in the $v$ variable only. It inspired H\"ormander and his theory of
hypoellipticity \cite{hormander}, where the regularisation is
recovered by more robust and more geometric commutator estimates (see
also \cite{rs}).

Another question which has attracted a lot of attention in calculus of
variations and partial differential equations along the 20th century
is Hilbert's 19th problem about the analytic regularity of solutions
to certain integral variational problems, when the quasilinear
Euler-Lagrange equations satisfy ellipticity conditions. Several
previous results had established the analyticity conditionally to some
differentiability properties of the solution, but the full answer came
with the landmark works of De~Giorgi \cite{DeG56,DeG} and
Nash \cite{nash}, where they proved that  any solution to these
variational problems with square integrable derivative is
analytic. More precisely their key contribution is the
following\footnote{We give the parabolic version due to Nash here.}:
reformulate the quasilinear parabolic problem as 
\begin{equation}
  \label{eq:2}
  \partial_t f = \nabla_v \left( A(v,t) \nabla_v f \right), \quad t \ge 0,
  \ v \in \R^d
\end{equation}
with $f = f(v,t) \ge 0$ and $A = A(v,t)$ satisfies the ellipticity
condition $0 < \lambda I \le A \le \Lambda I$ for two constants
$\lambda,\Lambda >0$ but is, besides that, merely measurable. Then the
solution $f$ is H\"older continuous.  \medskip

The method has been extended to degenerate cases, like the
$p$-Laplacian, first in the elliptic case by Ladyzhenskaya and
Uralt'seva \cite{L1}, and then, degenerate parabolic cases were
covered by DiBenedetto \cite{Di1} (see also DiBenedetto,
Gianazza and Vespri \cite{L2,L3,L4}). More recently, the method has
been extended to integral operators, such as fractional diffusion, in
\cite{CV1, CV2} --- see also the work of Kassmann \cite{K} and of
Kassmann and Felsinger \cite{K1}. Further application to fluid
mechanics can be found in \cite{V2,V3,V4}.

\subsection{Main results}

In view of the Landau equation and the nonlinear (quasilinear)
equation~\eqref{eq:1}, it is natural to ask whether a similar result as
the one of De~Giorgi-Nash holds for hypoelliptic equations. More
precisely, we consider the following kinetic Fokker-Planck equation
\begin{equation}
  \label{eq:main}
  \partial_t f + v \cdot \nabla_x f = \nabla_v \cdot \left(A \nabla_v
    f  \right) + B \cdot \nabla_v f + s, \quad t \in (0,T),\  (x,v) \in \Omega, 
\end{equation}
where $\Omega$ is an open set of $\R^{2d}$, $f = f(x,v,t)$, $B$ and
$s$ are bounded measurable coefficients depending in $(x,v,t)$, and
the $d \times d$ real matrices $A$, $B$ and source term $s$ are
measurable and satisfy
\begin{equation}\label{eq:ellipticity}
\begin{cases} 
& 0 < \lambda I \le A \le  \Lambda I \\
& |B| \le  \Lambda \\
& s \text{ essentially bounded}
\end{cases}
\end{equation} 
for two constants $\lambda,\Lambda$. We establish the H\"older
continuity of solutions to this problem. To state the result, we have
to define cylinders that respect two invariant transformations of the
(class of) equation(s): the scaling $(x,v,t) \mapsto (r^3x,rv,r^2t)$
and the transformation
\begin{equation}\label{eq:transformation}
  \mathcal{T}_{z_0} : z \mapsto (x_0+ x + t v_0,v_0+v, t_0+t).
\end{equation}
Given $z_0= (x_0,v_0,t_0) \in \R^{2d+1}$, the cylinder $Q_r(z_0)$
``centered'' at $z_0$ of ``radius'' $r$ is defined as
\begin{equation}
\label{def:slanted}
Q_r (z_0) = \left\{ (x,v,t): |x-x_0 - (t-t_0)v_0| < r^3, |v-v_0|< r, 
t \in \left(t_0-r^2,t_0\right]\right\}.
\end{equation}
When $z_0=0$, we shall omit to specify the base point:
$Q_r := Q_r(0,0,0)$.
 \medskip

 The \emph{weak solutions} to equation~\eqref{eq:main} on
 $U_x \times U_v \times I$, $U_x \subset \R^d$ open,
 $U_v \subset \R^d$ open, $I = [a,b]$ with $-\infty<a<b \le +\infty$,
 are defined as functions
 $f \in L^\infty_t(I,L^2_{x,v}(U_x \times U_v))) \cap L^2_{x,t}(U_x
 \times I, H^1_v(U_v))$ such that
 $\partial_t f + v\cdot \nabla_x f \in L^2_{x,t}(U_x \times I,
 H^{-1}_v(U_v))$ and $f$ satisfies the equation~\eqref{eq:main} in the
 sense of distributions.

\begin{theorem}[H\"older continuity]\label{thm:holder}
  Let $f$ be a weak solution of \eqref{eq:main} in $\Qext := Q_{r_0}(z_0)$ and
  $\Qint := Q_{r_1}(z_0)$ with $r_1 < r_0$. Then $f$ is $\alpha$-H\"older
  continuous with respect to $(x,v,t)$ in $\Qint$ and
  \[ \| f \|_{C^\alpha(\Qint)} \le C \left(\|f \|_{L^2(\Qext)} +
  \|s\|_{L^\infty(\Qext)}\right)\]
for some $\alpha$ universal
    (i.e. $\alpha=\alpha(d,\lambda,\Lambda)$) and $C =
    C(d,\lambda,\Lambda,\Qext,\Qint)$.
\end{theorem}
In order to prove such a result, we first prove that $L^2$
sub-solutions are locally bounded; we refer to such a result as an
$L^2-L^\infty$ estimate. We then prove that solutions are H\"older
continuous by proving a lemma which is an hypoelliptic counterpart of
De~Giorgi's ``isoperimetric lemma''.

We moreover prove a ``quantitative version'' of the strong maximum
principle: a Harnack inequality.
\begin{theorem}[Harnack inequality]\label{thm:harnack}
  If $f$ is non-negative weak solution of \eqref{eq:main} in $Q_1$,
  then
  \begin{equation}\label{eq:harnack}
 \sup_{Q^-} f \le C \left(\inf_{Q^+} f +\|s\|_{L^\infty(Q_1)} \right)
\end{equation}
where $Q^+ := Q_R$ and $Q^- := Q_R (0,0,-\Delta)$ and $C>1$ and
$R,\Delta \in (0,1)$ are small (in particular $Q^\pm \subset Q_1$ and
they are disjoint), and universal, i.e. only depend on dimension and
ellipticity constants.
\end{theorem}
\begin{remark}
  Using the transformation
  $\mathcal{T}_{z_0} (x,v,t)= (x_0+ x + t v_0, v_0+v,t_0+t)$, we get a
  Harnack inequality for cylinders centered at an arbitrary point
  $z_0=(x_0,v_0,t_0)$.
\end{remark}

\subsection{Comments and previously known results}

In \cite{pp}, the authors obtain an $L^2-L^\infty$ estimate with
completely different techniques; however they cannot reach the
H\"older continuity estimate. Our techniques rely on the velocity
averaging method. Velocity averaging designates a special type of
smoothing effect for solutions of the free transport equation
\[
(\partial_t+v\cdot\nabla_x)f=S
\]
observed for the first time in \cite{Agosh,gps} independently, later
improved and generalized in \cite{GLPS,DPL}. This smoothing effect
bears on averages of $f$ in the velocity variable $v$, i.e. on
expressions of the form
\[
\int_{\R^d}f(x,v,t)\,  \phi(v) \, {\rm d} v\,,
\]
say for $C^\infty_c$ test functions $\phi$. Of course, no smoothing on
$f$ itself can be observed, since the transport operator is hyperbolic
and propagates the singularities. However, when $S$ is of the form
\[
S=\nabla_v \cdot \left(A(x,v,t)\nabla_vf\right)+s
\]
where $s$ is a given source term in $L^2$, the smoothing effect of
velocity averaging can be combined with the $H^1$ regularity in the
$v$ variable implied by the energy inequality in order to obtain some
amount of smoothing on the solution $f$ itself. A first observation of
this type (at the level of a compactness argument) can be found in
\cite{PLLCam}. More recently, Bouchut \cite{bouchut} has obtained more
quantitative Sobolev regularity estimates. These estimates are one key
ingredient in our proof.

We give two proofs of this $L^2-L^\infty$ estimate, one following
Moser's approach, the other following De~Giorgi's ideas.  We emphasize
that, in both approaches, the main ingredient is a local gain of
integrability of non-negative sub-solutions. This latter is obtained
by combining a comparison principle and a Sobolev regularity
estimate. We then prove the H\"older continuity through a De~Giorgi
type argument on the decrease of oscillation for solutions. We also
derive the Harnack inequality by combining the decrease of oscillation
with a result about how the minimum of non-negative solutions
deteriorates with time.

In \cite{wz09,wz11}, the authors get a H\"older estimate for weak
solutions of so-called ultraparabolic equations, including
\eqref{eq:main}. Their proof relies on the construction of cut-off
functions and a particular form of weak Poincar\'e inequality
satisfied by non-negative weak sub-solutions.  Our paper proposes an
alternate method based on velocity averaging. It illustrates the
interesting connection between velocity averaging and
hypoelliptic-like structures. It also provides several tools for
further applications.

The $C^\infty$ smoothing of solutions to the Landau equation has been
investigated so far in two different settings: either for weak
spatially homogeneous solutions (non-negative in $L^1$ and with finite
energy) \cite{MR1055522,d04,molecules,dvI} (see also the related
entropy dissipation estimates in \cite{dv,d15}), or for classical
spatially heterogeneous solutions \cite{cdh,MR3191417}. The analytic
regularisation of weak spatially homogeneous solutions was
investigated in the case of Maxwellian or hard potentials in
\cite{MR2557895}. Let us also mention that in \cite{luislandau},
Silvestre derives an $L^\infty$ bound on the spatially homogeneous
solutions for soft potentials \emph{without relying on energy methods}
(which implies as well the smoothing by standard parabolic
techniques). Let us also mention works studying modified Landau
equations~\cite{MR2901061,MR2914961} and the work \cite{MR3599518}
that shows that any weak radial solution to the Landau-Coulomb
equation that belongs to $L^{3/2}$ is automatically bounded and $C^2$
using barrier arguments. Finally, we highlight the related results of
regularisation for the Boltzmann equation without long-range
interactions \cite{MR1324404,MR2820356,MR2885564}, and the related
perturbative results for the Landau and (long-range interaction)
Boltzmann equation
\cite{MR1946444,MR2784329,MR2679369,MR2795331,MR2556715,MR3158719,MR3375485}.
From this review, and the best of our knowledge, the regularity of
\emph{a priori} non-negative locally $L^\infty$ solutions (under our
assumption \eqref{e:meh}) to the spatially heterogeneous Landau
equation has not investigated so far.

A part of the results of this paper were announced in \cite{gv,im}. 

\subsection{Plan of the paper}
\label{sec:plan-paper}

In Section~\ref{sec:local-gain-integr}, we prove the universal gain of
integrability for non-negative sub-solutions. In Section~\ref{sec:up},
we derive from this gain of integrability a local upper bound of such
non-negative sub-solutions; we give two proofs: one following de
Giorgi's approach and the other one following Moser's iteration
procedure.  In Section~\ref{sec:osc-decrease}, the H\"older estimate
is derived by proving a lemma of ``reduction of oscillation''.  In
Section~\ref{sec:harnack} we prove a Harnack inequality for
non-negative solutions. In Section~\ref{sec:local-gain-regul}, we
prove a local gain of regularity of sub-solutions. In
Section~\ref{sec:l2+eps}, we prove that the velocity gradient of the
solution is slightly better than square integrable.

\subsection{Notation} We occasionally write $A \lesssim B$ in order to
say that $A \le \bar C B$ for some constant $\bar C$ which only
depends on dimension and ellipticity constants $\lambda$ and
$\Lambda$. Such a constant $\bar C$ is called \emph{universal}.

The inverse transformation $\mathcal{T}_{z_0}^{-1} : z \mapsto 
z_0^{-1} \circ z$ is defined by
\[ \mathcal{T}^{-1}_{z_0} (z) = (x-x_0 -(t-t_0)v_0,v-v_0,t-t_0).\]
The notation $z_0 \circ z$ and $z_0^{-1}$ refers to a Lie group
structure associated with the equation.

\section{Local gain of regularity / integrability}
\label{sec:local-gain-integr}

We consider the equation \eqref{eq:main} and we want to establish a
local gain of integrability of solutions in order to apply De
Giorgi-Moser's iteration and get a local $L^\infty$ bound. Since we
will need to perform convex changes of unknown, it is necessary to
obtain this gain for all (non-negative) \emph{sub-solutions}. The next
theorem is stated in cylinders centered at the origin.
\begin{theorem}[Gain of integrability for non-negative sub-solutions]\label{thm:gain}
  Consider two cylinders $\Qint :=Q_{r_1}$ and $\Qext :=Q_{r_0}$
  with $0< r_1 < r_0$.  There exists $p>2$ (only depending on dimension)
  such that for all non-negative sub-solution $f$ of \eqref{eq:main}
  in $\Qext$, we have
\begin{equation}
  \label{eq:gain}
  \n{f}_{L^p (\Qint)}^2 \le  \bar C \left( C_{0,1}^2  \n{f}_{L^2(\Qext)}^2 
 + C_{0,1} \int_{\Qext} |s|^2\un_{f>0} \right)
\end{equation}
with 
\[
C_{0,1} =  \left(\frac1{r_0^2 -r_1^2}+\frac{r_0}{r_0^3-r_1^3}+ \frac1{(r_0-r_1)^2}+ 1
\right) \quad \mbox{ and } \quad \bar C= \bar C (d,\lambda,\Lambda) \,
.
\]
\end{theorem}
\begin{remark}\label{rem:p} 
  The exponent $p$ is obtained by the Sobolev embedding
  $H^{\frac13} (\R^{2d+1}) \hookrightarrow L^p (\R^{2d+1})$, that is to say
  $p := 6(2d+1)/ (6d+1)$.
\end{remark}
This result is a consequence of the comparison principle and the
following gain of regularity.
\begin{theorem}[Gain of regularity for sign-changing solutions]\label{thm:gain-diff}
  Consider $z_0 \in \R^{2d+1}$ and two cylinders $\Qint :=Q_{r_1}(z_0)$
  and $\Qext:=Q_{r_0}(z_0)$ with $0<r_1 < r_0$.  Then any (sign-changing)
  weak solution $f$ of \eqref{eq:main} in $\Qext$ satisfies
\begin{equation}
  \label{eq:gain-diff}
  \n{f}_{H^{\frac13}_{x,v,t} (\Qint)}^2 \le C 
\left(  \n{f}_{L^2(\Qext)}^2 + \n{s}_{L^2(\Qext)}^2  \right)
\end{equation}
with $C =C(d,\lambda,\Lambda,\Qext,\Qint)$. 
\end{theorem}
\begin{remark}
  Using Theorem~\ref{thm:gain} and De~Giorgi-Moser's iteration, it is
  in fact possible to prove that this gain of regularity is also true
  for non-negative sub-solutions, as we will see in
  Section~\ref{sec:local-gain-regul}.
\end{remark}


\subsection{Global  estimates and  gain of regularity/integrability}
\label{subsec:global}

Remark that our weak solutions in
$f \in L^\infty_t(I,L^2_{x,v}(U_x \times U_v))) \cap L^2_{x,t}(U_x
\times I, H^1_v(U_v))$ are in
$C^0_t(I,L^2_{x,v}(U_x \times U_v) \cap H^{1/2}_t(I,L^2_{x,v}(U_x
\times U_v))$, following and adapting respectively the by-now standard
arguments in \cite{MR627929} and \cite{MR0338568} to the kinetic
case. This justifies the calculations performed in our energy estimates
in the sequel. 

\begin{lemma}[Global  estimate]\label{lem:global} 
Let $g$ be a weak solution  of
\[   (\partial_t  + v \cdot \nabla_x) g = \nabla_v \cdot (A \nabla_v g) + \nabla_v \cdot H_1 + H_0
\quad \text{ in } \quad \R^{2d+1} \]
with $H_1$ and $H_0$ in $L^2 (\R^{2d+1})$ and $g, H_0$ and $H_1$ supported in $\R^d \times B(0,r_0) \times \R$. 
Then
\begin{equation}
  \label{e:gain-txv} \|\nabla_v g \|^2_{L^2} + \| D^{\frac13}_x g \|_{L^2}^2 + \|D^{\frac13}_t g\|_{L^2}^2 \le C \bigg(
  \|H_1\|_{L^2}^2 +  \|H_0  \|_{L^2}^2 \bigg)
\end{equation}
where $C = \bar C (1+r_0^2)$ and $\bar C = \bar C (d,\lambda,\Lambda)$. 
In particular, there exists $p>2$ (only depending on dimension) such that 
\begin{equation}\label{e:gain-int}
 \| g \|_{L^p} ^{2} \le C \bigg( \|H_1\|_{L^2}^2 +\|H_0\|_{L^2}^2 \bigg) 
\end{equation}
where $C = \bar C (1+r_0^2)$ and
$\bar C = \bar C (d,\lambda,\Lambda)$.
\end{lemma}
\begin{proof}
Integrating against $2g$ in $\R^{2d+1}$ yields
\begin{multline*}
2 \lambda \int_{\R^{2d+1}}  |\nabla_v g |^2 \dd x \dd v \dd t \le 
 \int_{\R^{2d+1}} (-2 H_1 \cdot \nabla_v g +  2 g H_0) \dd x \dd v \dd t \\
 \le  \frac{\lambda}2 \int_{\R^{2d+1}} |\nabla_v g |^2 \dd x \dd v \dd t 
       + \frac2{\lambda} \int_{\R^{2d+1}} | H_1|^2 \dd x \dd v \dd t +
       2 \int_{\R^{2d+1}}  |g| |H_0| \dd x \dd v \dd t .  
\end{multline*}
Moreover
\[ 2 \int_{\R^{2d+1}} |g||H_0| \dd x \dd v \dd t \le \eps
\int_{\R^{2d+1}} |g|^2 \dd x \dd v \dd t + \frac1{\eps}
\int_{\R^{2d+1}} |H_0|^2 \dd x \dd v \dd t .\]
Since $g$ is supported in $B(0,r_0)$ in the velocity variable, we can
use the Poincar\'e inequality to get
\[
\eps \int_{\R^{2d+1}} |g|^2 \dd x \dd v \dd t \le C_P r_0^2 \eps \int_{\R^{2d+1}}
|\nabla_v g|^2 \dd x \dd v \dd t
\]
and we choose $\eps$ such that $C_P r_0^2 \eps = \lambda/2$.  This implies
\begin{equation}\label{e:grad g}
 \|\nabla_v g \|_{L^2}^2  \le   C \left( \| H_1\|_{L^2}^2+     \| H_0 \|_{L^2}^2
 \right) .   
\end{equation}

Applying \cite[Theorem~1.3]{bouchut} with $p=2$, $r=0$, $\beta =1$,
$m=1$, $\kappa=1$ and $\Omega =1$ yields
\begin{align*}
 \| D^{\frac13}_x g \|_{L^2}^2 + \|D^{\frac13}_t g\|_{L^2}^2  \lesssim  \ & \|g\|_{L^2}^2 
 + \| \nabla_v g \|_{L^2} \|(1+|v|^2)^{\frac12}  H_0  \|_{L^2} \\
& +   \| \nabla_v g \|_{L^2}^{\frac43} \|   (1+|v|^2) (H_1 +A \nabla_v g) \|_{L^2}^{\frac23} \\
& +   \| \nabla_v g \|_{L^2} \| (1+|v|^2)^{\frac12} ( H_1 +A \nabla_v g )\|_{L^2} .
\end{align*}
Using the fact that $g$, $H_0$ and $H_1$ are supported in
$\R^d \times B(0,r_0) \times \R$, we get
\begin{align*}
  \| D^{\frac13}_x g \|_{L^2}^2 + \|D^{\frac13}_t g\|_{L^2}^2  \lesssim \ & r_0^2   \|\nabla_v g\|_{L^2}^2 
  + (1+r_0^2)^{\frac12} \| \nabla_v g \|_{L^2} \|  H_0  \|_{L^2} \\
& \phantom{r_0^2   \|\nabla_v g\|_{L^2}^2}
+  (1+r_0^2)^{\frac23} \| \nabla_v g \|_{L^2}^{\frac43} \left(\| H_1 \|_{L^2}^{\frac23} +\| \nabla_v g \|_{L^2}^{\frac23}\right) \\
& \phantom{r_0^2   \|\nabla_v g\|_{L^2}^2} +   (1+r_0^2)^{\frac12} \| \nabla_v g \|_{L^2} \left(\|  H_1 \|_{L^2} +\| \nabla_v g \|_{L^2} \right) \\
\lesssim  \ &  (1+r_0^2) \left( \|\nabla_v g\|_{L^2}^2 + \|H_1\|_{L^2}^2 + \|\nabla_v g\|_{L^2} \|H_0\|_{L^2} \right).
\end{align*}
Combining this estimate with \eqref{e:grad g} yields
\eqref{e:gain-txv}. The proof is now complete.
\end{proof}

\subsection{The local energy estimate}

\label{subsec:reverse}

The gain of integrability with respect to $v$ and $t$ is classical; it
derives from the natural energy estimate, after truncation. We follow
here \cite{moser}.
\begin{lemma}[The local energy estimate]\label{lem:caccio}
  Under the assumptions of Theorems~\ref{thm:gain} and
  \ref{thm:gain-diff}, any sub-solution $f$ satisfies
\begin{equation}
\label{eq:h1v-loc}
\sup_{t} \int_{\Qint^t} f^2 (\cdot,\cdot,t)  + 
\int_{\Qint}  |\nabla_v f|^2  \le \bar C \left(
  C_{0,1} \int_{\Qext} f^2  + \int_{\Qext}  |s|^2 \right)
\end{equation}
for $\Qint^t := \{ (x,v) \in \R^{2d}: (x,v,t) \in \Qint\}$,  $\bar C = \bar C (d,\lambda,\Lambda)$ and
\[ C_{0,1} = \left(
  \frac1{r_0^2-r_1^2}+\frac{r_0}{r_0^3-r_1^3}+\frac1{(r_0-r_1)^2}  +1 \right).\]
Moreover, if the sub-solution $f$ is non-negative, then
\begin{equation}
\label{eq:h1v-loc-bis}
\sup_{t} \int_{\Qint^t} f^2 (\cdot,\cdot,t) + 
 \int_{\Qint}  |\nabla_v f|^2  \le \bar C \left(
   C_{0,1} \int_{\Qext} f^2 + \int_{\Qext} |s|^2
   \un_{f>0} \right) .
\end{equation}
\end{lemma}
\begin{proof}
  Consider $\Psi \in C^\infty_c(\R^{2d} \times \R)$ with
  $0 \le \Psi \le 1$ and integrate the inequation satisfied by $f$
  against $2 f \Psi^2$ in $\cR:=\R^{2d} \times [t_1,0]$ with
  $t_1 \in (-r_1^2,0]$ and get
\begin{equation*} 
\int_{\cR} \partial_t (f^2) \Psi^2 +  \int_{\cR} v \cdot \nabla_x (f^2) \Psi^2 
\le 2 \int_{\cR} \nabla_v \cdot ( A \nabla_v f) f \Psi^2 + 2 \int_{\cR}  (B \cdot \nabla_v f ) f \Psi^2  
+ 2\int_{\cR} fs \Psi^2 .
\end{equation*}
Add $\int_{\cR} f^2 \partial_t (\Psi^2)$, integrate by parts and use
the upper bound on $A$ to get
\begin{align*}
 & \int_{\cR} \partial_t (f^2 \Psi^2) +  2 \lambda \int_{\cR} |\nabla_v f |^2
\Psi^2 \\
 & \quad \le   \int_{\cR} f^2 (\partial_t + v \cdot \nabla_x )(\Psi^2) 
-4 \int_{\cR}  \Psi A \nabla_v f \cdot  f  \nabla_v \Psi  
+ 2 \int_{\cR} (B \cdot \nabla_v f) f \Psi^2 +2\int_{\cR} fs \Psi^2 \\
& \quad \le \int_{\cR} f^2 (\partial_t + v \cdot \nabla_x )(\Psi^2) 
 +  4 \Lambda  \int_{\cR} ( |\nabla_v f |\Psi)f( \Psi+|\nabla_v \Psi|) +  2\int_{\cR} fs \Psi^2 \\
& \quad  \le  \int_{\cR} f^2 \Big[(\partial_t + v \cdot \nabla_x )(\Psi^2) + 8 \Lambda^2 \lambda^{-1}(|\nabla_v \Psi|^2 
+ \Psi^2) \Big] + 2 \int_{\cR} fs \Psi^2 +  \lambda \int_{\cR} |\nabla_v f|^2 \Psi^2.
\end{align*}
We thus get 
\begin{equation} \label{eq:nrj-estimate}
  \int_{\cR} \partial_t (f^2 \Psi^2) + \lambda \int_{\cR} |\nabla_v f|^2
  \Psi^2  
\le \bar{C} \left( \n{\partial_t \Psi}_{\infty} + r_0 \n{\nabla_x \Psi}_{\infty} +
    \n{\nabla_v \Psi}_{\infty}^2 + 1 \right)   \int_{\cR \cap \supp \Psi} f^2 + 2\int_{\cR} fs \Psi^2 
\end{equation}
with $\bar{C} = \bar C(d,\lambda,\Lambda)$.  Choose next $\Psi^2$ such
that $\Psi (t=0)=0$ and $\supp \Psi \subset \Qext$ and get for $t_1
\in \R$:
\[ \int_{\R^{2d}} f^2(\cdot,\cdot,t_1)  \Psi^2 (t_1) \dd x \dd v + \lambda
\int_{\R^{2d+1}} |\nabla_v f|^2 \Psi^2 \dd x \dd v \dd t
\le  C   \int_{\Qext} f^2  + 2 \int_{\Qext} |f| \, |s| . \]

If $\Psi$ additionally satisfies $\Psi \equiv 1$ in $\Qint$, we get
\eqref{eq:h1v-loc}. Remark that \eqref{eq:h1v-loc-bis} is a simple consequence
of \eqref{eq:h1v-loc}. 
 The proof is now complete. 
\end{proof}

\subsection{Local gain: proofs}

\begin{proof}[Proof of Theorems~\ref{thm:gain} and \ref{thm:gain-diff}]
  We first remark that if $f$ is a non-negative sub-solution of
  \eqref{eq:main}, then $f=f \un_{f \ge 0}$ and it is also a
  sub-solution of the same equation when the source term $s$ is
  replaced with $s \un_{f \ge 0}$.

  For $i=1,\frac{1}2$, consider $f_i = f \chi_i$ where $\chi_1$ and
  $\chi_{1/2}$ are two truncation functions such that
 \begin{align*}
  \chi_1 \equiv 1 \text{ in } \Qint  & \quad \text{ and } \quad 
 \chi_1 \equiv 0 \text{ outside } \Q12 , \\[2mm]
  \chi_{\frac12} \equiv 1 \text{ in } \Q12 & \quad \text{ and } \quad 
 \chi_{\frac12} \equiv 0 \text{ outside } \Qext \,.
   \end{align*}  
The function $f_1$ now satisfies 
\[   (\partial_t  + v \cdot \nabla_x) f_1 \le 
\nabla_v \cdot (A \nabla_v f_1) + \nabla_v \cdot H_1 +  H_0
\quad \text{ in }\quad  \R^{2d+1}\]
with $ H_1$ and $ H_0$ given by 
\[\begin{cases}
  H_1 &= (- A  \nabla_v \chi_1) f_\frac12  \\[2mm]
  H_0 &= (B\chi_1 -A \nabla_v \chi_1 ) \cdot \nabla_v f_\frac12 +
  \alpha_1 f_{\frac12} +s \un_{\{ f \ge 0\}}\chi_1
\end{cases}\]
with $\alpha_1 = (\partial_t + v \cdot \nabla_x ) \chi_1$. We remark
that $f_1$, $H_0$ and $H_1$ are supported in $\Qext$.

We now consider the solution $g$ of 
\[ (\partial_t + v \cdot \nabla_x) g = \nabla_v \cdot (A \nabla_v g) +
\nabla_v \cdot H_1 + H_0 \quad \text{ in }\quad \R^{2d+1}.\]
We remark that $g$ is also supported in $\Qext$, and since
$h:= f_1 -g$ is a sub-solution of the equation
$\partial_t h +v \cdot \nabla_x h \le \nabla_v ( A \nabla_v h)$ with
zero initial data at $t = -r_0^2$, the comparison principle implies
that $h \le 0$ everywhere, and therefore $0 \le f_1 \le g$. It can be
proved for instance by observing that $h_+$ is also a sub-solution of
the same inequation and the standard energy estimate implies that its
$L^2_{x,v}$-norm is non-increasing along the time variable.

Moreover,
\begin{align*}
\begin{cases}
\ds \|H_1\|_{L^2}^2 \lesssim &   \|\nabla_v \chi_1\|_{L^\infty}^2 \|f\|^2_{L^2(\Qext)}\\[3mm]
\ds \| H_0 \|_{L^2}^2 \lesssim &  \left( 1+ \|\nabla_v \chi_1\|_{L^\infty}^2 \right) \|\nabla_v f \|^2_{L^2(\Q12)}  + 
\|\alpha_1\|_{L^\infty}^2  \|f\|_{L^2(\Qext)}^2  + \| s \un_{\{ f \ge0\}}\|^2_{L^2(\Qext)}.
\end{cases}
\end{align*}
In view of Lemma~\ref{lem:caccio}, we know that
\[
 \|\nabla_v f \|^2_{L^2(\Q12)} \lesssim C_{0,1} \|f\|^2_{L^2(\Qext)} +  \| s \un_{\{ f \ge 0\}} \|^2_{L^2(\Qext)}.
\]
Hence,
\begin{multline*}
 \| H_0\|_{L^2}^2 + \|H_1\|_{L^2}^2 \lesssim   \Big[ ( 1+ \|\nabla_v \chi_1\|_{L^\infty}^2)(1+C_{0,1})
+ \|\alpha_1\|_{L^\infty}^2 \Big] \|f\|_{L^2(\Qext)}^2 \\
+ \left(2+\|\nabla_v \chi_1\|_{L^\infty}^2\right) \| s \un_{\{f\ge 0\}} \|^2_{L^1(\Qext)}.
\end{multline*}
In view of the definition of $C_{0,1}$ in Lemma~\ref{lem:caccio}, we thus get 
\[ \| H_0\|_{L^2}^2 + \|H_1\|_{L^2}^2 \lesssim C_{0,1}^2 \|f\|_{L^2(\Qext)}^2 
+ (r_0-r_1)^{-2} \| s \un_{\{f\ge 0\}} \|^2_{L^1(\Qext)}.\]
Lemma~\ref{lem:global} then yields 
\[  \n{g}_{L^p (\Qint)}^2 \le  \bar C \left( C_{0,1}^2  \n{f}_{L^2(\Qext)}^2 
 + C_{0,1} \int_{\Qext} |s|^2\un_{f\ge 0} \right).\]
We then obtain \eqref{eq:gain} by using the fact that $0 \le f_1 \le g$. 
This achieves the proof of Theorem~\ref{thm:gain}. 

As for Theorem~\ref{thm:gain-diff}, Lemma~\ref{lem:global} can be
applied directly to $f_1$ and the conclusion follows along the same
lines, with some simplifications.
\end{proof}

\section{Local upper bounds for non-negative sub-solutions}
\label{sec:up}

In this section, we prove that non-negative $L^2$ sub-solutions are in
fact locally bounded.
\begin{theorem}[Upper bounds for non-negative $L^2$
  sub-solutions] \label{thm:sup-sub} Given two cylinders
  $\Qext:=Q_{r_0}(z_0)$ and $Q_\infty:=Q_{r_\infty} (z_0)$ with
  $0<r_\infty<r_0$, let $f$ be a non-negative $L^2$ sub-solution of
  \eqref{eq:main} in $\Qext$ with $s \in L^q (\Qext)$ and
  $q>(2p)/(p-1)$ with $p$ only depending on dimension.  There for any
  $\mathfrak g >$, there exists
  $\kappa= \kappa(d,\lambda,\Lambda, \Qext,Q_{\infty},\mathfrak g,q)
  >0$ such that
  \[ 
  \left\{
    \begin{array}{r} 
      \|s\|_{L^q(\Qext)} \le \mathfrak g \\[2mm]
      \|f\|_{L^2(\Qext)} \le \kappa 
    \end{array} \right\}
  \qquad \Rightarrow \qquad f \le \frac12 \text{ in } Q_\infty.
  \]
\end{theorem}
\begin{remark}
  The exponent $p = 6(2d+1)/ (6d+1)$ is the one given by the gain of
  integrability in Theorem~\ref{thm:gain} (see Remark~\ref{rem:p}).
\end{remark}
We give two proofs of such a result. The first one sticks to the case
$q=+\infty$ with no lower order terms and use Moser's approach. The
second one deals with the general case and use De~Giorgi's approach.

\subsection{Moser's approach}

\begin{proof}[Proof of Theorem~\ref{thm:sup-sub} in the case without
  source term by Moser's iteration.]
  Using tranformations introduced in Eq.~\eqref{eq:transformation}, we
  reduce to the case $z_0=0$.

We first observe that, for all $q>1$,
  the function $f^q$ satisfies
\[ (\partial_t + v\nabla_x) f^q \le \nabla_v \cdot (A \nabla_v f^q) \quad
\text{ in } Q_{r_0} .\]

We now rewrite \eqref{eq:gain} with $s=0$ from $Q_{r_n}$ to
$Q_{r_{n+1}}$ with $r_{n+1} < r_n$ as follows:
\begin{equation}\label{eq:gain-again}
  \left( \int_{Q_{r_{n+1}}(0)} (f^q)^p \right)^{\frac2p} \le \bar C 
  C_n^2  \int_{Q_{r_n}(0)} f^{2q}  
\end{equation}
where $\bar C = \bar C (d,\lambda,\Lambda)$ and
\begin{equation}\label{eq:cq}
 C_n =  \left(
   \frac1{r_n^2-r_{n+1}^2}+\frac{r_n}{r_n^3-r_{n+1}^3}+\frac1{(r_n-r_{n+1})^2}
 \right) + \|B\|_{L^\infty}+1. 
\end{equation}
Choose now $q = q_n = (p/2)^n$ for $n \in \N$ 
and write  
$a_n$ for
$(\int_{Q_n} f^{2q_n})^{1/(2q_n)}$.  Using that for
$\bar C=\bar C (d,\lambda,\Lambda,\Qext)\ge 1$ large enough, we have
$|\Qext| \le \bar C$, we get from \eqref{eq:gain-again}
\begin{equation}\label{eq:an}
a_{n+1}  \le (\bar C)^{\frac1{2q_n}}  \left( C_n \right)^{\frac1{q_n}} a_n.
\end{equation}
Finally we choose 
\[r_{n+1} = r_n - \frac{1}{a(n+1)^2}\]
for some $a>0$ (only depending on $r_0-r_\infty$) so that
\eqref{eq:cq} yields $C_n \sim a^2 n^4$ as $n \to +\infty$. In
particular, we can choose
$\bar C = \bar C (d,\lambda,\Lambda,\|B\|_{L^\infty})$ large enough so
that $C_n \le \bar C^{\frac12} a^2 n^4$ and we get from
\eqref{eq:an} that
\[ a_{n+1} \le (\bar C a^2 n^4)^{\frac1{q_n}} a_n.\]
The convergence of the following infinite product
\[ \prod_{n =0}^\infty  (\bar C a^2)^{\frac1{q_n}} (n^4)^{\frac1{q_n}} < + \infty\]
achieves the proof. 
\end{proof}

\subsection{De~Giorgi's approach}
\label{sec:deGup}

\begin{proof}[Proof of Theorem~\ref{thm:sup-sub} by De~Giorgi's
  approach.]
  We again reduce to the case $z_0=0$ thanks to the transformation
  $\mathcal{T}_{z_0}^{-1}$ defined in Eq.~\eqref{eq:transformation}.
  For $n \ge 0$ integer, consider radius $r_n$, time $T_n$, cylinder
  $Q_n$ and constant $C_n$ as follows
\[ r_n =  r_\infty + (r_0-r_\infty)2^{-n}, \quad  T_n = t_0-r_n^2,
\quad C_n = \frac12 (1-2^{-n}), \]
and cut-off functions $\Psi_n$ (independent of time) as follows
\[ \Psi_n \equiv \begin{cases} 
  1 & \text{ in } Q_{r_n}^0 \\[3mm] 
  0 & \text{ outside } Q_{r_{n-1}}^0 
\end{cases} 
\quad \text{ and } \quad \begin{cases} 
\| \ds \nabla_v \Psi_n \|_{L^\infty} \le \frac{1}{r_{n-1}-r_n} \le C_{0,\infty}  2^n \\[3mm]
\| \ds \nabla_x \Psi_k \|_{L^\infty} \le \frac{1}{r_{n-1}^3 - r_n^3} \le C_{0,\infty}  2^n
\end{cases}
\]
where $C_{0,\infty} = C(r_0,r_\infty)$ only depends on $r_0$ and
$r_\infty$, and as before
\[
Q^\tau_r := \{ (x,v) : (x,v,\tau) \in Q_{r}\} .
\]
\smallskip

\noindent \textbf{The energy estimate.}  Remark that $f_n = (f-C_n)^+$
is a sub-solution of \eqref{eq:main} in $Q_{r_n}$ with
$s_n = s \un_{f\ge C_n}$.  Then the energy
estimate~\eqref{eq:nrj-estimate} obtained in the proof of
Lemma~\ref{lem:caccio} yields for all
$T_{n-1} \le \tau \le T_n \le t \le 0$,
\begin{equation} \label{eq:nrj-k-1} 
  \int_{Q_{r_n}^{t}} f_n^2 
  + \lambda \int_{Q_{r_n}} |\nabla_v f_n|^2
  \le \int_{Q_{r_n}^{\tau}} f_n^2   + \left(r_n \n{\nabla_x
      \Psi_n}_{\infty} + \n{\nabla_v \Psi_n}_{\infty}^2 + 1 \right)
  \int_{Q_{r_{n-1}}} f_n^2 + 2\int_{Q_{r_{n-1}}} f_n |s|.
\end{equation}
Averaging both sides of the inequality in $\tau \in (T_{n-1},T_n)$
and using the estimates on the gradients of the cut-off function yields
\begin{equation} 
  \label{eq:nrj-k-2} 
  U_n := \sup_{ t \in (T_n,0)}
  \int_{Q_{r_n}^{t}} f_n^2 \le C 4^n \int_{Q_{r_{n-1}}} f_n^2 +
  2\int_{Q_{r_{n-1}}} f_n |s|
\end{equation}
where $C = C (r_0,r_\infty)$. 
Remark that, 
\begin{equation}\label{eq:uk}
U_n \le U_{n-1} \le \dots \le U_0 \le \kappa  \le 1
\end{equation}
(we choose $\kappa \le 1$). 
\medskip 

\noindent \textbf{The non-linearization procedure.}  Using the
(universal) exponent $p>2$ given by Theorem~\ref{thm:gain}, we next
estimate the terms in the right hand side of \eqref{eq:nrj-k-2} as
follows
\begin{align}\label{eq:source}
\begin{cases}
\ds \int_{Q_{r_{n-1}}} f_n^2 &\le \left( \int_{Q_{r_{n-1}}} f_n^p \right)^{\frac2p} |\{ f_n \ge 0 \} \cap Q_{r_{n-1}}|^{1 -\frac2p} \\[3mm]
\ds \int_{Q_{r_{n-1}}} f_n |s| &\le \mathfrak g \left( \int_{Q_{r_{n-1}}} f_n^p
\right)^{\frac1p} |\{ f_n \ge 0 \} \cap Q_{r_{n-1}}|^{1 -\frac1p-\frac1q} 
\end{cases}
\end{align}
(we used that $\|s\|_{L^q(\Qext)} \le \gamma$) 
if $p$ and $q$ satisfy 
\[1 - \frac1p - \frac1q >0.\] 
We next remark that $\{ f_n \ge 0 \} = \{f_{n-1} \ge C_n-C_{n-1} = 2^{-k-1}\}$ which in turn implies
\begin{equation}\label{eq:ls}
|\{ f_n \ge 0 \} \cap Q_{r_{n-1}}|  \le 2^{2n+2} \int_{Q_{r_{n-1}}} f_{n-1}^2   \le \bar C 4^n U_{n-1}. 
\end{equation}
Combining these three estimates with \eqref{eq:nrj-k-2} yields 
\begin{equation}\label{eq:nrj-2}
 U_n \le  C 2^{4n} \left[ \left( \int_{Q_{r_{n-1}}} f_{n-1}^p \right)^{\frac2p} U_{n-1}^{1-\frac2p} + 
\|s\|_{L^q(\Qext)} \left( \int_{Q_{r_{n-1}}} f_{n-1}^p \right)^{\frac1p} U_{n-1}^{1-\frac1p-\frac1q} \right] 
\end{equation}
(we also used that $f_n \le f_{n-1}$) where $C = C(d,\lambda,\Lambda, r_0,r_\infty)$. 
\medskip 

\noindent \textbf{Use of the gain of integrability.} 
In view of Theorem~\ref{thm:gain}, we know that 
\[
\left( \int_{Q_{r_{n-1}}} f_{n-1}^p \right)^{\frac2p} \le C \left( 8^n
  \int_{Q_{r_{n-2}}} f_{n-1}^2 +4^n \int_{Q_{r_{n-2}}}
  s^2\un_{f_{n-1}>0} \right)
\]
with $C= C(d,\lambda,\Lambda,r_0,r_\infty)$.  We next estimate the
terms in the right hand side of the previous equation depending of the
source term as in \eqref{eq:source} but with $p=2$: we use
\eqref{eq:ls} to get
\begin{align*}
  \int_{Q_{r_{n-2}}} s^2 \un_{f_{n-1} \ge 0} 
  \le  \mathfrak g^2 |\{ f_{n-1} > 0 \} \cap Q_{r_{n-2}} |^{1-\frac{2}q} 
  \le    \mathfrak g^2 2^{2n-\frac{4n}q} U_{n-2}^{1-\frac2q}.
\end{align*}
Hence, we can use \eqref{eq:uk} and $U_0 \le 1$ again in order to write
\begin{align*}
  \left( \int_{Q_{r_{n-1}}} f_{n-1}^p \right)^{\frac2p}  \le C
  \left(2^{3n} U_{n-2} 
  + 2^{4n-\frac{4n}q} U_{n-2}^{1-\frac2q} \right) \le C 2^{4n} U_{n-2}^{1-\frac2q}
\end{align*}
with $C= C(d,\lambda,\Lambda,r_0,r_\infty,q,\mathfrak g)$.  Then
\eqref{eq:nrj-2} and \eqref{eq:uk} imply
\[ 
U_n \le C 2^{4n} \left( 2^{4n} U_{n-2}^{2 - \frac2p -\frac2q} +
U_{n-2}^{\frac32 - \frac1p -\frac2{q}}\right) \le C 2^{8n} U_{n-2}^{\frac32
  - \frac1p -\frac2q}.
\]
\medskip

\noindent \textbf{Conclusion.}
Remark that we can assume that $C \ge 1$. We rewrite it as 
\begin{equation}\label{eq:nl}
 V_n \le \beta^n V_{n-1}^\alpha 
\end{equation}
where $V_n = U_{2n}$, $\beta =  2^8C$ and $\alpha = \frac32 - \frac1p -\frac2q$. 
Remark that $\alpha >1$ as soon as 
\[ \frac1q < \frac12 \left(\frac12- \frac1p \right).\]

Applying \eqref{eq:nl} recursively, we get 
\[ V_n \le \beta^{k+ \alpha (k-1)+ \alpha^2 (k-2) + \dots + \alpha^{k-1}} V_0^{\alpha^k}.\]
Remark now that 
\begin{align*}
 n + \alpha (n-1) + \dots + \alpha^{n-1} 
& = n(1+\alpha+\dots + \alpha^{n-1}) - \alpha (1+2\alpha + \dots + (n-1)\alpha^{n-2}) \\
& = n\frac{\alpha^n-1}{\alpha-1} - \alpha \frac{{\rm d}}{{\rm d}\alpha}\left( \frac{\alpha^n-1}{\alpha-1} \right)\\
& = \frac{n(\alpha^n-1)}{\alpha-1} - \alpha \left( \frac{n\alpha^{n-1}(\alpha-1) -(\alpha^n-1)}{(\alpha-1)^2}\right)\\
& = \frac{\alpha(\alpha^n-1) -n(\alpha-1)}{(\alpha-1)^2} \le \frac{\alpha}{(\alpha-1)^2} \alpha^n. 
\end{align*}
Hence
\[ V_n \le \left( \beta^{\frac\alpha{(\alpha-1)^2}} V_0 \right)^{\alpha^n}.\]
This implies that $U_{2n} = V_n \to 0$ as $n \to +\infty$ as soon as 
\[ \beta^{\frac\alpha{(\alpha-1)^2}} V_0 \le  (2^8 C)^{\frac\alpha{(\alpha-1)^2}} \kappa <1 \]
where $C = C(d,\lambda,\Lambda,r_0,r_\infty,q,\mathfrak g)$. Hence,
\[ U_\infty = \int_{Q_{r_\infty}} \left(f-\frac12\right)_+^2 = 0 \]
which means that $f \le 1/2$ in $Q_{r_\infty}$. This completes the proof
of Theorem~\ref{thm:sup-sub}.
\end{proof}

\section{Intermediate-value lemma and H\"older continuity}
\label{sec:osc-decrease}

\subsection{A De~Giorgi intermediate-value lemma}

An important step in the proof of regularity in De~Giorgi's method for
elliptic equations is based on an inequality of isoperimetric form
(see the proof of \cite[Lemma~II]{DeG}). This inequality is a
quantitative variant of the well-known fact that no $H^1$ function can
have a jump discontinuity, and can also be understood as a
quantitative minimum principle. More precisely, given an $H^1$
function $u$ valued in $[0,1]$ and which takes the values $0$ and $1$
on sets of positive measure, De~Giorgi's isoperimetric inequality
provides a lower bound on the measure of the set of intermediate
values $\{ 0<u<1 \}$. In the present subsection, we establish an
analogue of this inequality adapted to our equation and the
combination of the first order transport operator and the second order
elliptic operator in the velocity variable. 

We prove the core lemma at ``unit scale''.  We recall that
$Q_2= B_8 \times B_2 \times (-4,0]$ and
$Q_1 = B_1 \times B_1 \times (-1,0]$,
$Q_\omega = B_{\omega^3} \times B_\omega \times (-\omega^2,0]$ and we
denote the shifted cube
$\hat Q := Q_\omega (0,0,-1)=B_{\omega^3} \times B_\omega \times
(-1-\omega^2,-1]$ (see Figure~\ref{fig:isoperim}). 

\begin{figure}
\includegraphics[height=6cm]{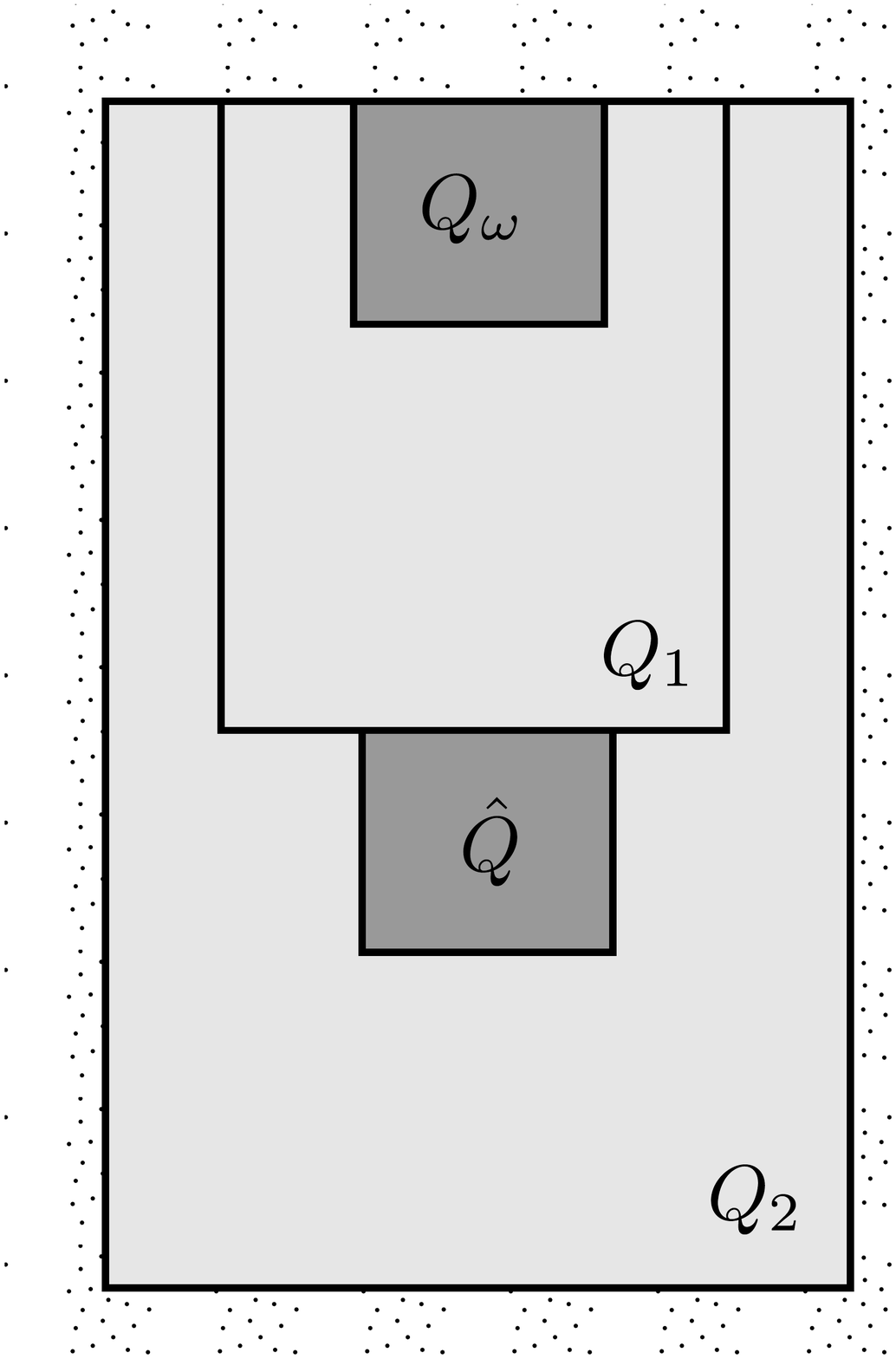}
\caption{Cylinders involved in the statement of the De~Giorgi
  intermediate-value Lemma.}
\label{fig:isoperim}
\end{figure}
\begin{lemma}[A De~Giorgi intermediate-value lemma]\label{lem:isoperim}
  Let $\omega =\frac14$.  For any (universal)
  constants $\delta_1 \in (0,1)$, $\delta_2 \in (0,1)$ there exist
  $\nu>0$ and $\theta \in (0,1)$ (both universal) such that for any
  sub-solution $f$ of \eqref{eq:main} in $Q_2$ with
\[ f \le 1 \quad \text{ and } \quad |s|\le 1\]
  and
\[\begin{aligned}
|\{ f \ge 1-\theta \} \cap Q_\omega | &\ge  \delta_1 |Q_\omega|\\
|\{ f \le 0 \} \cap \hat Q | & \ge \delta_2 |\hat Q|
\end{aligned}
\]
we have
\[ |\{ 0 < f < 1-\theta\} \cap B_1 \times B_1 \times (-2,0]| \ge \nu.\]
\end{lemma}
\begin{remark}
  While De~Giorgi's isoperimetric inequality is based on an explicit
  computation leading to a precise estimate with effective constants,
  the proof of Lemma~\ref{lem:isoperim} is obtained by an argument by
  contradiction, so that the values of $\theta$ and
  $\nu$ are not known explicitly.
\end{remark}
\begin{remark}
  The compactness argument used in the proof is reminiscent of one
  used by Guo in \cite{guo} and of one used by the fourth author in
  \cite{vasseur}.
\end{remark}
\begin{proof}
  We argue by contradiction by assuming that there exists a sequence
  $(f_k)_{k \ge 0}$ of sub-solutions:
\begin{equation}\label{eq:fk}
(\partial_t + v\cdot \nabla_x) f_k \le \nabla_v \cdot (A_k \nabla_v f_k) + B_k \cdot \nabla_v f_k + s_k
\end{equation}
such that $f_k \le 1$ and $|s_k| \le 1$ and 
\[ \begin{cases} \theta_k \to 0 \\[1mm] \alpha_k \to 0 \end{cases} \quad \text{ as } \quad k \to +\infty \]
and
\[
\begin{aligned}
|\{ f_k \ge 1 -\theta_k \} \cap Q_\omega | &\ge  \delta_1 |Q_\omega| \\
|\{ f_k \le 0 \} \cap \hat Q | & \ge \delta_2 |\hat Q|\\
|\{ 0 < f_k < 1 -\theta_k\} \cap (Q_1 \cup \hat Q)| & \to 0 \quad \text{ as } \quad k \to +\infty. 
\end{aligned}
\]
The convexity of $z \mapsto z^+$ together with $|s_k|\le 1$ implies
that the non-negative part $f_k^+$ of $f_k$ satisfies the same
inequation, and therefore
\begin{equation}\label{eq:fk+}
(\partial_t + v\cdot \nabla_x) f_k^+ = \nabla_v \cdot (A_k \nabla_v f_k^+)+ B_k \cdot \nabla_v f_k^+ + 1 - \mu_k
\end{equation}
for some non-negative measures $\mu_k$. 
\smallskip

\noindent
\textbf{A priori estimates for $f_k^+$.}  The natural energy estimate
is obtained by multiplying the equation with $f_k^+ \Psi^2$ with a
smooth cut-off function $\Psi$ supported in $Q_2$ and valued in
$[0,1]$, and using the fact that $f_k^+ \le 1$ and $|s_k|\le 1$:
\begin{align*}
 \lambda \int_{\R^{2d+1}} | \nabla_v f_k^+ |^2  \Psi^2 &\le 
\bar C \int_{\R^{2d+1}} \left(\Psi^2 + |\nabla_v \Psi|^2 + \Psi|(\partial_t + v\cdot \nabla_x)\Psi|\right) 
+ \Lambda \int_{\R^{2d+1}} |\nabla_v f_k^+| f_k^+ \Psi^2  \\
& \le \bar C \int_{\R^{2d+1}} \left(\Psi^2 + |\nabla_v \Psi|^2 + \Psi|(\partial_t + v\cdot \nabla_x)\Psi|\right) 
+ \frac{\lambda}2 \int_{\R^{2d+1}} |\nabla_v f_k^+|^2  \Psi^2.  
\end{align*}
Hence
\begin{equation}\label{estim:1}
 \lambda \int_{\R^{2d+1}} | \nabla_v f_k^+ |^2  \Psi^2 \le 
\bar C \int_{\R^{2d+1}} \left(\Psi^2 + |\nabla_v \Psi|^2 + \Psi|(\partial_t + v\cdot \nabla_x)\Psi|\right) 
\end{equation}
where $\bar C=\bar C (d,\lambda,\Lambda)$. 

We can also multiply the equation by $\Psi^2$ and get
\begin{multline*}
-  \int_{\R^{2d+1}} f_k^+  (\partial_t + v \cdot \nabla_x) (\Psi^2) = - 
 \int_{\R^{2d+1}} A_k \nabla_v f_k^+  \cdot \nabla_v (\Psi^2) + 
\int_{\R^{2d+1}} B_k \cdot \nabla_v f_k^+ \Psi^2 \\ +  \int_{\R^{2d+1}}  \Psi^2
- \int_{\R^{2d+1}} \Psi^2 \dd \mu_k . 
\end{multline*}
Combining the latter equation with \eqref{estim:1}, we deduce
\begin{equation}\label{estim:2}
  \int_{\R^{2d+1}} \Psi^2 \dd \mu_k  \le \bar C \int_{\R^{2d+1}} 
  \left(\Psi^2 + |\nabla_v \Psi|^2 + \Psi|(\partial_t + v\cdot \nabla_x)\Psi|\right)
\end{equation}
where  $\bar C=\bar C (d,\lambda,\Lambda)$. 
\medskip

\noindent \textbf{Passage to the limit.}  On the one hand,
Banach-Alaoglu theorem implies that 
\[ f_k^+ \stackrel{\ast}{\rightharpoonup} F  \text{ in } L^\infty_{\text{loc}} (Q_2) \]
and 
\begin{equation}\label{conv:grad}
 \nabla_v f_k^+ \rightharpoonup \nabla_v F \quad \text{ and } \quad
\begin{cases}
 A_k \nabla_v f_k^+  \rightharpoonup H_1 \\[2mm]
B_k \cdot \nabla_v f_k^+ \rightharpoonup H_0 
\end{cases}
\text{ in } L^2_{\text{loc}} (Q_2) 
\end{equation}
for some weak limit
$F \in L^\infty_{\text{loc}} (Q_2) \cap
(L^2_{x,t}H^1_v)_{\text{loc}}(Q_2)$.
In particular, \eqref{estim:1} implies that
\begin{equation}\label{eq:H1}
 \int_Q |\nabla_v F|^2 \lesssim_Q 1 
\end{equation}
for all $Q \Subset Q_2$, with a control depending on $Q$.  On the
other hand, the bound \eqref{estim:2} implies that
\[ \mu_k \rightharpoonup \mu \quad \text{ in } \mathcal{M} (Q_2).\]
We thus have
\begin{eqnarray}
\label{eq:F}
 (\partial_t + v \cdot \nabla_x)  F =   \nabla_v H_1 + H_0 + 1 - \mu.
\end{eqnarray}

By velocity averaging (see Theorem 1.8 in \cite{BGP}), together with
the bound \eqref{estim:1}, we deduce the strong convergence
\[ f_k^+ \to F \text{ in } L^p_{\text{loc}}(Q_2) \quad \text{
  for } 1 \le p < +\infty.\] 
It implies the convergence in probability and thus the
function $F$ satisfies 
\begin{align}
\label{eq:P1} |\{ F = 1 \} \cap Q_\omega | &\ge  \delta_1 |Q_\omega|\\
\label{eq:P2} |\{ F = 0 \} \cap \hat Q | & \ge \delta_2 |\hat Q| \\
\nonumber \left|\{ 0 < F < 1\} \cap \left( B_1 \times B_1 \times
  (-2,0] \right) \right| & = 0 \,.
\end{align}
In view of \eqref{eq:H1}, since indicator functions are not in $H^1$
unless they are constant, we have that for almost every
$(x,t) \in B_1 \times (-1,0)$,
\[ \left\{
\begin{aligned}
  \text{either } \quad  &  \mbox{for almost every }  v \in B_1, 
  \quad F(x,v,t) = 0 \\[1mm]
  \text{ or } \quad \quad & \mbox{for almost every } v\in B_1, 
  \quad F(x,v,t) = 1.
\end{aligned}
\right.
\]
In other words, $F (x,v,t) = \un_P (x,t)$ for some measurable set $P
\subset B_1 \times (-1,0)$.  In view of \eqref{eq:P1} and
\eqref{eq:P2}, $P$ satisfies
\begin{equation}
\label{eq:P3}
\begin{cases}
 |P \cap B_{\omega^3} \times (-\omega^2,0)| >0 \\[2mm]
 |B_{\omega^3} \times (-1-\omega^2,-1) \setminus P| >0.
\end{cases}
\end{equation}
\smallskip

\noindent \textbf{Propagation.} 
We thus get from \eqref{eq:F}
\[ \partial_t F + v \cdot \nabla_x F \le \nabla_v H_1 + H_0 + 1 \quad
\text{ in } B_1 \times B_1 \times (-2,0).\]
Consider a cut-off funtion $\xi \in \mathcal{D} (\R^d)$ such that
\[ 
\int_{\R^d} \zeta(z) \dd z = 1, \quad \zeta (z) = \zeta (-z), \quad \supp \zeta
\subset B_{\frac12}. 
\] 
Given $v_0 \in B_{\frac12}$, since $F$ only depends on $(t,x)$, we can use a
test-function of the form $\zeta(v-v_0)$, and get for all $v_0 \in B_{\frac12}$, 
\begin{align*}
 \partial_t F + v_0 \cdot \nabla_x F 
\le \int_{\R^d} \Big[ \left|H_1 (x,v,t) \nabla_v \zeta(v-v_0)\right| 
+ \left|H_0(x,v,t) \zeta (v-v_0)\right| \Big] \dd v +  1 
\end{align*}
in $(x,t) \in B_1 \times (-2,0)$.  Since $F$ is an indicator function
and $H_0,H_1 \in L^2_{\text{loc}} (Q_2)$, this implies for
$v_0 \in B_{\frac12}$,
\begin{equation}\label{eq:propagation}
 \partial_t F + v_0 \cdot \nabla_x F \le 0 \quad \text{ in } B_1
 \times (-2,0) .
\end{equation} 
We next remark that 
\begin{equation}
\label{eq:connection}
\begin{cases}
\text{ for all } (x,t) \in B_{\omega^3} \times
(-\omega^2,0) \quad \text{ and } \quad (x_0,t_0) \in
  B_{\omega^3} \times (-1-\omega^2,-1), \\[2mm]
\text{ there exists } v_0 \in B_{\omega} \mbox{
    so that } (x_0,v_0,t_0) \in \hat Q   \text{
  and }  (x,t) = (x_0 + s v_0,t_0+s).
\end{cases}
\end{equation}
Indeed, the time shift $s$ is fixed by $t=t_0+s$ and belongs to
$(1-\omega^2,1+\omega^2)$. Then the velocity $v_0$ is fixed by $x = x_0 + s
v_0$ and  satisfies 
\[ |v_0| = \frac{|x-x_0|}{t-t_0} < \frac{2\omega^3}{1-\omega^2} \le
\omega\]
since $\omega =\frac14 \le \frac1{\sqrt{3}}$.  Since
$|B_1 \times (-1-\omega^2,-1) \setminus P| >0$ (see \eqref{eq:P3}), we
can use \eqref{eq:propagation} and \eqref{eq:connection} and conclude
that $F \equiv 0$ in $Q_\omega$, and contradicts \eqref{eq:P3}.  The
proof is complete.
\end{proof}

\subsection{Improvement of oscillation}

It is classical that H\"older continuity is a consequence of the
decrease of the oscillation of the solution ``at unit scale''. 
\begin{lemma}[Improvement of oscillation]\label{lem:osc-decrease}
  There exist $\lambda_0 \in (0,1)$, $\omega \in (0,1/2)$ and
  $\beta >0$ (all universal) such that any $f$ solution of
  \eqref{eq:main} in $Q_2$ with $\osc_{Q_2} f \le 2$ and
  $|s| \le \beta$ satisfies
\[ \osc_{Q_{\frac\omega2}} f \le 2 - \lambda_0. \]
\end{lemma}
This lemma is a consequence of the following one. 
\begin{lemma}[A measure-to-pointwise estimate]\label{lem:measure-to-pointwise}
  Given $\delta_2 >0$, there exist $\lambda_0 \in (0,1)$,
  $\omega \in (0,1/2)$ and $\beta >0$ (depending on $\delta_2$ but not
  on the sub-solution) such that any $f$ sub-solution of \eqref{eq:main}
  in $Q_2$ with $f \le 1$ and $|s| \le \beta$ such that
  $|\{ f \le 0 \} \cap \hat Q | \ge \delta_2|\hat Q|$ satisfies
  \begin{equation}\label{eq:sup-dec}
    f \le 1-\lambda_0 \quad \text{a.e. in } \quad Q_{\frac\omega2}.
  \end{equation}
\end{lemma}
\begin{proof}[Proof of Lemma~\ref{lem:osc-decrease}]
  Let $f$ be a solution of \eqref{eq:main} in $Q_2$ with
  $\osc_{Q_2} f \le 2$ and $|s|\le \beta$. We can reduce to the case
  where $|f| \le 1$. Indeed, we remark that there exists a constant
  $C$ such that $\tilde f = f -C$ satisfies \eqref{eq:main} in
  $Q_2(0)$ with $|\tilde f| \le 1$ and the same source term.

  If $|\{ f \le 0 \} \cap \hat Q | \ge |\hat Q|/2$, then apply
  Lemma~\ref{lem:measure-to-pointwise} with
  $\delta_2 = 1/2$.

  In the other case, considering $-f$ implies that the essential
  infimum of $f$ is raised. In both cases, we get the desired
  improvement of the oscillation of $f$. This completes the proof of
  the lemma.
\end{proof}
We now turn to the proof of Lemma~\ref{lem:measure-to-pointwise}. 
\begin{proof}[Proof of Lemma~\ref{lem:measure-to-pointwise}]
The proof proceeds in several steps. \medskip

\noindent \textbf{Choice of parameters.}  Theorem~\ref{thm:sup-sub}
provides us with $\kappa$ correponding to the upper bound $\mathfrak g =1$
on the source term and $\Qext = Q_\omega$ and
$Q_\infty = Q_{\frac\omega2}$.  Lemma~\ref{lem:isoperim} applied with
$\delta_2$ and $\delta_1 = \sqrt{\kappa}/|Q_\omega|$ provides us with
$\nu$ and $\theta$ universal.  We choose next $k_0$ the smallest
positive integer such that
\[ k_0 \nu > |B_1 \times B_1 \times (-2,0)|.\]
We finally choose $\beta$ such that $\beta \le \theta^{k_0}$.
\medskip

\noindent \textbf{Iteration.}
We  define $f_0 =f$ and 
\[ f_{k+1} = \frac1\theta (f_k -(1-\theta)) = \theta^{-k}
(f-(1-\theta^k)) .\] They satisfy $f_k \le 1$ and
\[ (\partial_t + v \cdot \nabla_x ) f_k \le  \nabla_v \cdot (A \nabla_v f_k) +B \cdot \nabla_v f_k + s_k\]
with $s_k = \theta^{-k} s$. In particular $|s_k|\le \theta^{-k_0}
\beta \le 1$ which allows to apply Theorem~\ref{thm:sup-sub} with the
upper bound $\mathfrak g=1$ as above. 
Remark that 
\begin{equation}\label{eq:estim-set}
  |\{ f_0 \le 0 \} \cap \hat Q| \ge \delta_2 |\hat Q| 
\quad \mbox{ and } \quad \{ f_{k+1} \le 0 \}  \supset \{ f_k \le 0 \}. 
\end{equation}

Our goal is to prove that there exists at least one index
$k \in \{1,\dots,k_0\}$ such that
\[ |\{f_{k} \ge 0 \} \cap Q_\omega | \le \delta_1 |Q_\omega| .\]
Indeed, remarking that for such an index $k_1$
\[ \|(f_{k_1})_+\|_{L^2(Q_\omega)} \le \bigg[ \big|\{ f_{k_1} \ge 0 \}
\cap Q_\omega \big| \bigg]^{\frac12} \le \sqrt{\delta_1 |Q_\omega|}
\le \kappa,\] Theorem~\ref{thm:sup-sub} then implies that
\[ f \le 1 - \frac12\theta^{k_1} \le 1- \frac12\theta^{k_0} \quad
\text{ in } \quad Q_{\frac\omega2}\]
which concludes the proof. 

Let us prove the claim by contradiction. Assume that for all
$k = 1,\dots,k_0$,
\[ |\{f_k \ge 0 \} \cap Q_\omega | \ge \delta_1 |Q_\omega|.\]
Since $f_{k+1} = \frac1\theta( f_k -(1-\theta))$, this also implies
for $k=0,\dots,k_0-1$,
\[ |\{f_k \ge 1 -\theta \} \cap Q_\omega | \ge \delta_1 |Q_\omega|.\]
But \eqref{eq:estim-set} also implies that for all $k \ge 0$,
\[ |\{ f_k \le 0 \} \cap \hat Q|  \ge \delta_2 | \hat Q|.\]
Hence Lemma~\ref{lem:isoperim} implies that for $k=0,\dots,k_0-1$,
\[ \left|\{ 0 \le f_k \le 1-\theta \} \cap \left( B_1 \times B_1
    \times (-2,0) \right) \right| \ge \nu .\]

Now remark that 
\begin{align*}
\left|\{f_{k+1} \le 0 \} \cap \left( B_1 \times B_1 \times (-2,0)
  \right) \right| 
= &  \left|\{f_k \le 0 \} \cap \left( B_1 \times B_1 \times (-2,0)
    \right) \right| \\
& + \left|\{0 \le f_k \le 1 -\theta \} \cap \left( B_1 \times B_1
  \times (-2,0) \right) \right| \\
\ge &  \left|\{f_k \le 0 \} \cap \left( B_1 \times B_1 \times (-2,0)
      \right) \right| + \nu.
\end{align*}
In particular 
\[
\left|B_1 \times B_1 \times (-2,0) \right|  \ge  \left|\{f_{k_0} \le 0
  \} \cap \left( B_1 \times B_1 \times (-2,0) \right) \right|  \ge  k_0 \nu
\]
which is impossible for $k_0$ as chosen above. The proof is now complete. 
\end{proof}

\subsection{Proof of the H\"older estimate}

\begin{proof}[Proof of Theorem~\ref{thm:holder}]
  Consider an $L^2$ solution $f$ of Eq.~\eqref{eq:main} in a cylinder
  $\Qext=Q_{r_0}(z_0)$.  By Theorem~\ref{thm:sup-sub}, we know that
  $f$ is locally bounded in $\Qext$.  In particular, $f$ is bounded in
  $\Q12 = Q_{\frac{r_0+r_1}2}(z_0)$ and
  \[ \|f\|_{L^\infty (\Q12)} \le C_0 \left( \|f \|_{L^2 (\Qext)} + \|
    s \|_{L^\infty(\Qext)} \right) \]
  for some constant $C_0 = C(d,\lambda,\Lambda,\Qext,\Q12)$.  If
  $f \equiv 0$ in $\Qext$, there is nothing to prove. If $f$ is not
  identically $0$, recalling that $\beta$ is given by
  Lemma~\ref{lem:osc-decrease}, we assume that
\[
 \|f \|_{L^\infty (\Q12)} \le 1 \quad \text{ and } \quad \|s\|_{L^\infty(\Qext)} \le \beta
\]
by considering, if necessary,
\[ 
\tilde{f} = \frac{f}{C_0\left( \|f \|_{L^2 (\Qext)} + \| s
    \|_{L^\infty(\Qext)} \right) + \beta^{-1}
  \|s\|_{L^\infty(\Qext)}}.
\]
\medskip

Let $z_1 \in \Qint := Q_{r_1}(z_0)$. We want to prove that for all
$r>0$ such that $Q_{2r}(z_1) \subset \Q12$,
\begin{equation}\label{eq:osc-alg}
 \osc_{Q_r(z_1)} f \le C r^\alpha
\end{equation}
for some universal $\alpha \in (0,1)$ and some constant
$C=C(d,\lambda,\Lambda,r_0,r_1)$. Let $\tilde r>0$ denote the largest
$r \in (0,1)$ such that $Q_{2r}(z_0) \subset \Q12$.  We remark that
for $r \in (0,\tilde r)$,
\( Q_{2r}(z_1) = \mathcal{T}^{-1}_{z_1} (Q_{2r}) \)
where $\mathcal{T}_{z_1}$ is defined in Eq.~\eqref{eq:transformation}
and $\bar f = f \circ \mathcal{T}_{z_1}$ satisfies \eqref{eq:main} in
$Q_{2\tilde r}$ with the source term
$\bar s := s \circ \mathcal{T}_{z_1}$ and the coefficients
$\bar A := A \circ \mathcal{T}_{z_1}$ and
$\bar B := B \circ \mathcal{T}_{z_1}$. In particular $\bar f$ and
$\bar s$ satisfy
\[
 \|\bar f \|_{L^\infty (Q_{2 \tilde r})} \le 1 \quad \text{ and }
 \quad \|\bar s\|_{L^\infty(Q_{2 \tilde r})} \le \beta
\]
 and \eqref{eq:osc-alg} is equivalent to: for all $r \in (0,\tilde r)$, 
\begin{equation}\label{eq:osc-alg-0}
 \osc_{Q_r} \bar f \le C r^\alpha .
\end{equation} 

We recall how to scale solutions. For all $r \in (0,\tilde r)$, the
function
\[\bar f_r (x,v,t)=  \bar f(r^3x,r v,r^2 t)\]
is defined in $Q_2$ and satisfies \eqref{eq:main} with
\[ \begin{cases}
\bar B_r (x,v,t) = r \bar B (r^3 x,r v,r^2 t)\\[2mm]
\bar s_r (x,v,t)= r^2 \bar s (r^3 x, r v,r^2 t). 
\end{cases}
\] 
Since $\osc_{Q_{2\tilde r}} \bar f \le 2$, we have
$\osc_{Q_2} \bar f_{\tilde r} \le 2$ and Lemma~\ref{lem:osc-decrease}
implies that
\[ \osc_{Q_{\frac\omega2}} \bar f_{\tilde r} = \osc_{Q_{\frac{\omega}2
    \tilde r}} \bar f \le 2 \theta \]
with $\theta = 1 - \lambda_0/2$ (we used the fact that
$\tilde r \le 1$ to ensure that
$\| \bar s_{\tilde r} \|_{L^\infty(Q_2)} \le \beta$).  We remark that
we can assume that $\theta \ge 1/2$ and we recall that
$\omega \in (0,1/2)$. We next apply Lemma~\ref{lem:osc-decrease} to
$\theta^{-1} \bar f_{\tilde r_1}$ with
$\tilde r_1 = (\omega/4) \tilde r$, which rescales the $L^\infty$ bound
on the source term by a factor $(\omega/4)^2 \theta^{-1}<1$ as compared to $\|
\bar s_{\tilde r} \|_{L^\infty(Q_2)} \le \beta$. Hence the bounds
assumed are still valid and we get
\[ \osc_{Q_{\tilde r_2}} \bar f \le 2 \theta^2 \]
with $\tilde r_2 = (\omega/2) \tilde r_1$. Inductively, we deduce
that
\[ \osc_{Q_{\tilde r_k}} \bar f \le 2 \theta^k \]
with $\tilde r_k = (\omega/2)^k \tilde r/2$.
This yields \eqref{eq:osc-alg-0} for $r =\tilde r_k$ with 
\[ 
\alpha = \frac{\ln \theta}{\ln (\omega/2)} \quad \text{ and } \quad C
= 2 \left(\frac2{\tilde r}\right)^\alpha.
\] 
If now $r \in [\tilde r_{k+1},\tilde r_k]$, then
\[ \osc_{Q_r} \bar f \le \osc_{Q_{r_k}} \bar f \le C \tilde r_k^\alpha
= C \left(\frac2\omega\right)^\alpha \tilde r_{k+1}^\alpha \le \tilde
C r^\alpha \]
with $\tilde C = C (2/\omega)^\alpha$. Observe finally that the
constant $C$ and $\tilde C$ are uniformly bounded above as $z_0$
varies in $\Qint$ since $\tilde r \ge r_1 - r_0$. The proof is now
complete.
\end{proof}

\section{Harnack inequality}
\label{sec:harnack}

In this section, we derive Harnack inequality for solutions to
Eq.~\eqref{eq:main}.  We use here an approach that Luis Silvestre
explained to us in the stationary setting: we start with H\"older
continuous solutions and we consider expanding cylinders to control
the spreading of the lower bound of non-negative solutions (see
Lemma~\ref{lem:doubling}).  The Harnack inequality is a consequence of
the decrease of oscillation we proved earlier and a so-called
``doubling property'' that estimates how the minimum of a solution
propagates with time. Let us first recall the decrease of oscillation
proposition.
\begin{proposition}[Decrease of oscillation]\label{prop:osc-decrease-true}
  There exist $\delta \in (0,1)$ and $\omega \in (0,1/2)$ (both
  universal) such that for any $r \in (0,1)$ and any solution $f$ of
  \eqref{eq:main} in some cylinder $Q_{2r}(z)$ satisfies
  \[ \osc_{Q_{\frac\omega4 r} (z)} f \le (1 - \delta) \left(
    \osc_{Q_r(z)} f + 2\beta^{-1} \|s\|_{L^\infty}\right). \]
\end{proposition}
\begin{remark}
The conclusion of the proposition is equivalent to 
\[ \osc_{Q_{\frac\omega4 r}} f \circ \mathcal{T}_z \le (1 -
\delta)\left( \osc_{Q_r} f \circ \mathcal{T}_z +
2\beta^{-1}\|s\|_{L^\infty}\right)\]
with $\mathcal{T}_z (y,w,s) = (x+ y + s v, v+w,t+s)$ where
$z=(x,v,t)$.
\end{remark}
\begin{proof}
  By considering
  \[
  \tilde f = \frac{ f \circ \mathcal{T}_z}{\osc_{Q_{2r}(z)} f/2 +
    \|s\|_{L^\infty}/\beta},
  \]
  and a rescaling $\tilde f_r$, we can assume that $z=0$ and
  $\osc_{Q_{2}} \tilde f_r \le 2$ and $\|s\|_{L^\infty} \le \beta$ (we
  use here that $r \le 1$). We then apply Lemma~\ref{lem:osc-decrease}
  to $\tilde f_r$ and get the desired result with
  $1-\delta = 1-\lambda_0/2$.
\end{proof}

\subsection{How minima propagate with time} 

The  goal of this subsection is to prove the following proposition. 
In order to state it, we introduce two cylinders which contain $Q^-$:
\[ Q^- \subset Q^-[1] \subset Q^-[2] \subset Q_1.\]
See Figure~\ref{cylinders}.
\begin{figure}
\includegraphics[height=6cm]{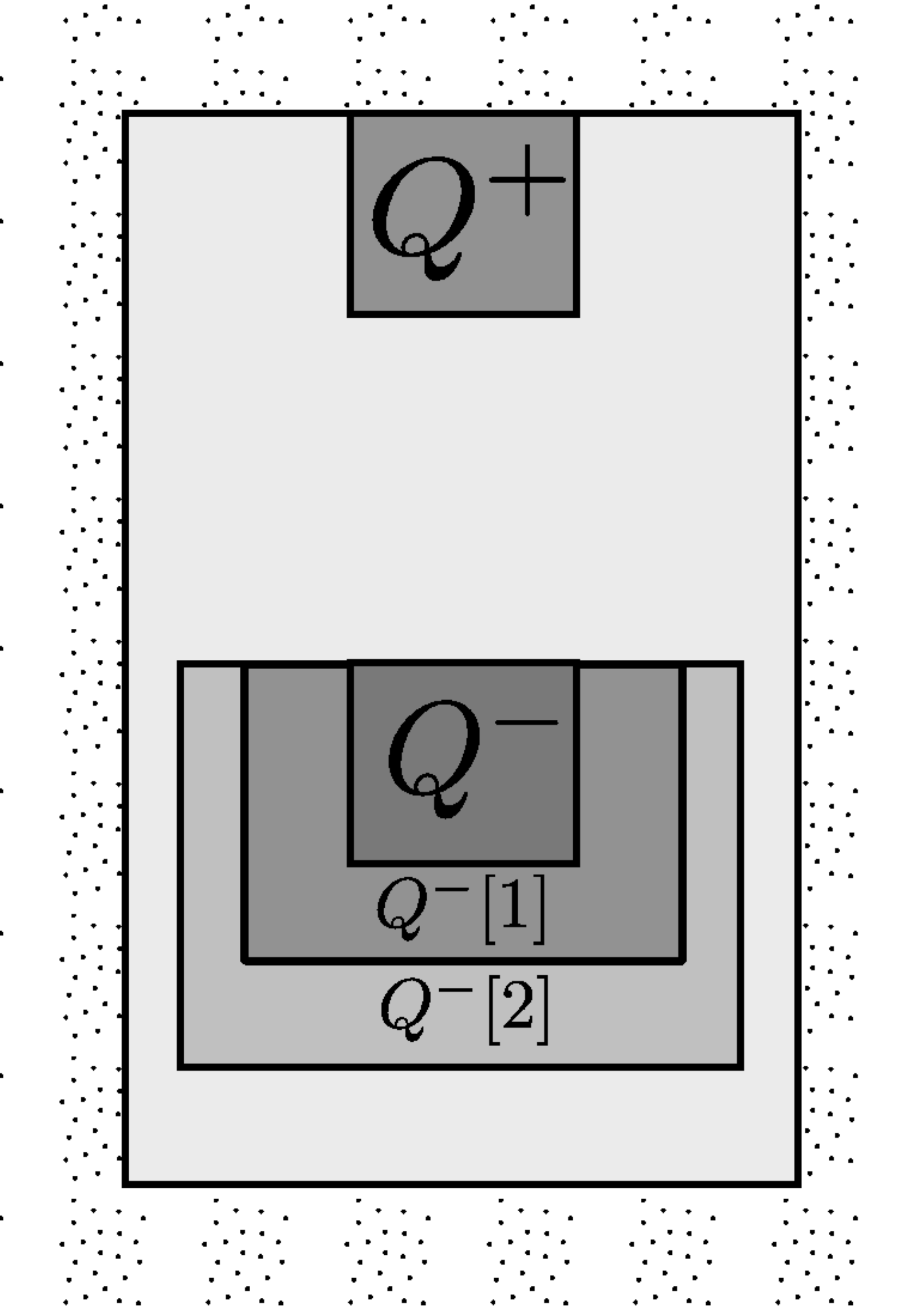}
\caption{The cylinders $Q^+$, $Q^-$, $Q^-[1]$ and $Q^-[2]$. Harnack
  inequality relates the supremum of a solution over $Q^-$ and its
  infimum over $Q^+$.  The proof consists in constructing a sequence
  of points lying in $Q^-[1]$ and whose corresponding values
  explode. Neighborhoods of points included in $Q^-[2]$ are also considered.}
\label{cylinders}
\end{figure}
We recall that $Q^+ = Q_R$ and $Q^- = Q_{R}(0,0,-\Delta)$ and
$R,\Delta \in (0,1)$ are small so that in particular
$Q^\pm \subset Q_1$ and they are disjoint.  We let $Q^-[i]$ be equal
to $Q_{\rho_i} (0,0,-\Delta)$ with $R < \rho_1 < \rho_2 < 1$.

In the following propositions, we introduce \emph{elongated} cylinders
$\Qel$ where the time is stretched longer in the past than what the
scaling would induce:
\begin{align*}
\Qel_1 &= B_{(\omega/4)^3} \times B_{\omega/4} \times (-1,0] \\
 \Qel_r (z) &= \mathcal{T}_z  \big( B_{(\omega/4)^3 r^3} \times B_{(\omega/4) r} \times (-r^2,0] \big).
\end{align*}
\begin{proposition}[The propagation of minima]\label{prop:propagation}
  Assume that $f$ is a non-negative super-solution of \eqref{eq:main}
  in $Q_1$  with a non-negative source term $s$.  There exists
  $r_0 >0$, $R>0$ (universal) such that for any $r \in (0,r_0)$ and
  $z \in Q^-$ such that $\Qel_r(z) \subset Q^-[2]$, we have
\[ \min_{\Qel_r (z)} f \le C_{\text{{\em pm}}} \; r^{-q} \min_{Q^+} f \]
for some universal constants $C_{\text{{\em pm}}}$ and $q >0$. 
\end{proposition}

We first derive from Lemma~\ref{lem:measure-to-pointwise} the
following doubling property at the origin. For the two next lemmas, it
is easier that $0$ is the final time of the first cylinder.
\begin{lemma}[The doubling property at the origin]\label{lem:doubling-0}
  There exists $\mathfrak h \in (0,1)$ (universal) such that for any
  non-negative super-solution $f$ of \eqref{eq:main} in
  $B_8 \times B_2 \times (-1,4]$ with $s \ge 0$, we have
  \[ \inf_{Q^1} f \ge \mathfrak h \inf_{Q^0} f\]
  with $Q^1 = \Qel_2(0,0,4)$ and $Q^0=\Qel_1$.
\end{lemma}
\begin{proof}
  We first note that since $s \ge 0$, the function $f$ is a
  super-solution of \eqref{eq:main} with $s=0$. 
We first prove that 
\begin{equation}\label{eq:min-delay}
  \inf_{Q_{\omega/2} (0,0,1)} f \ge \mathfrak h_0 \inf_{Q_{\omega/4}} f
\end{equation}
for some universal constant $\mathfrak h_0$; see Figure~\ref{fig:doubling}. 

\begin{figure}[ht]
\includegraphics[height=6cm]{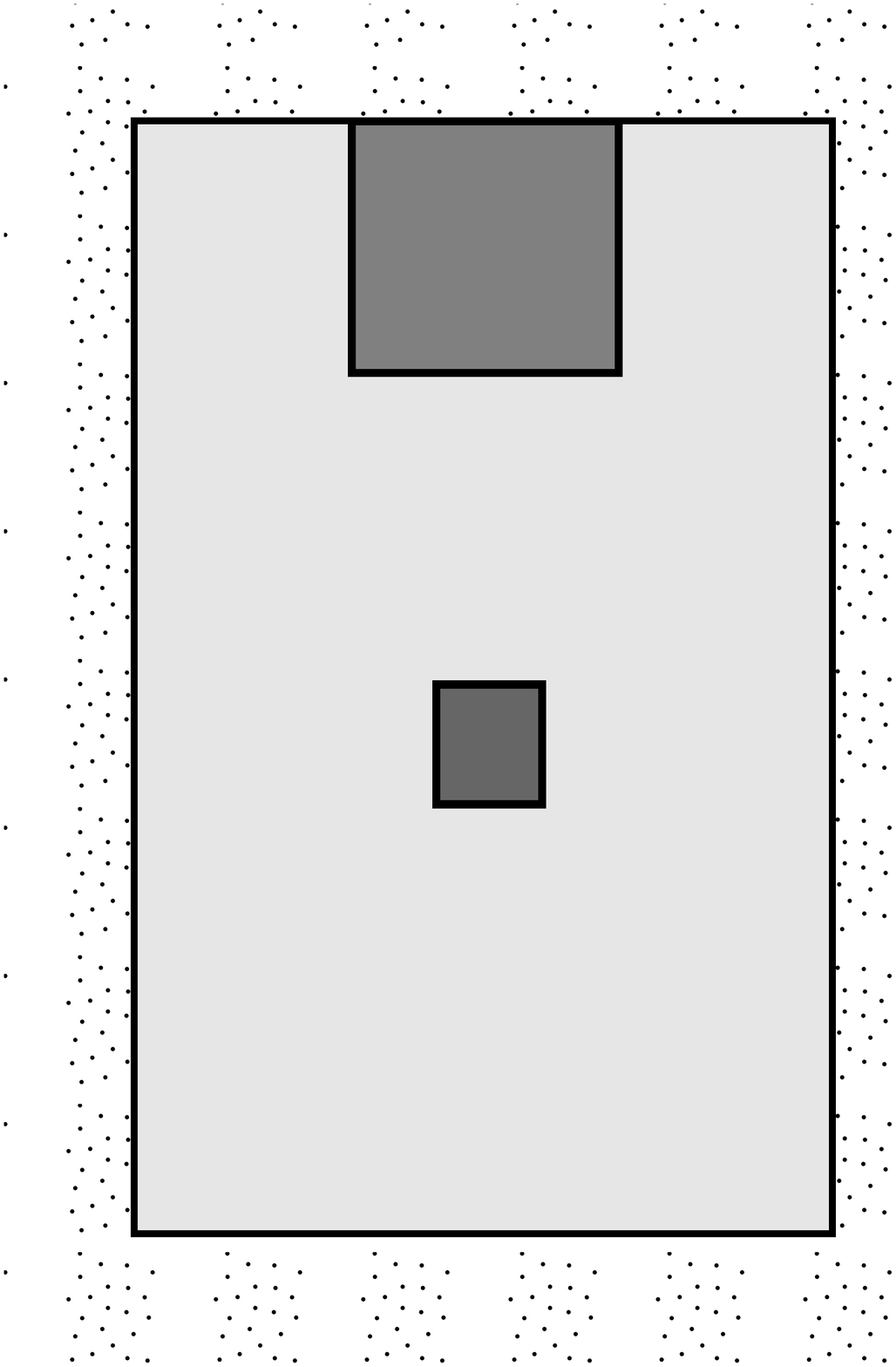}
\includegraphics[height=6cm]{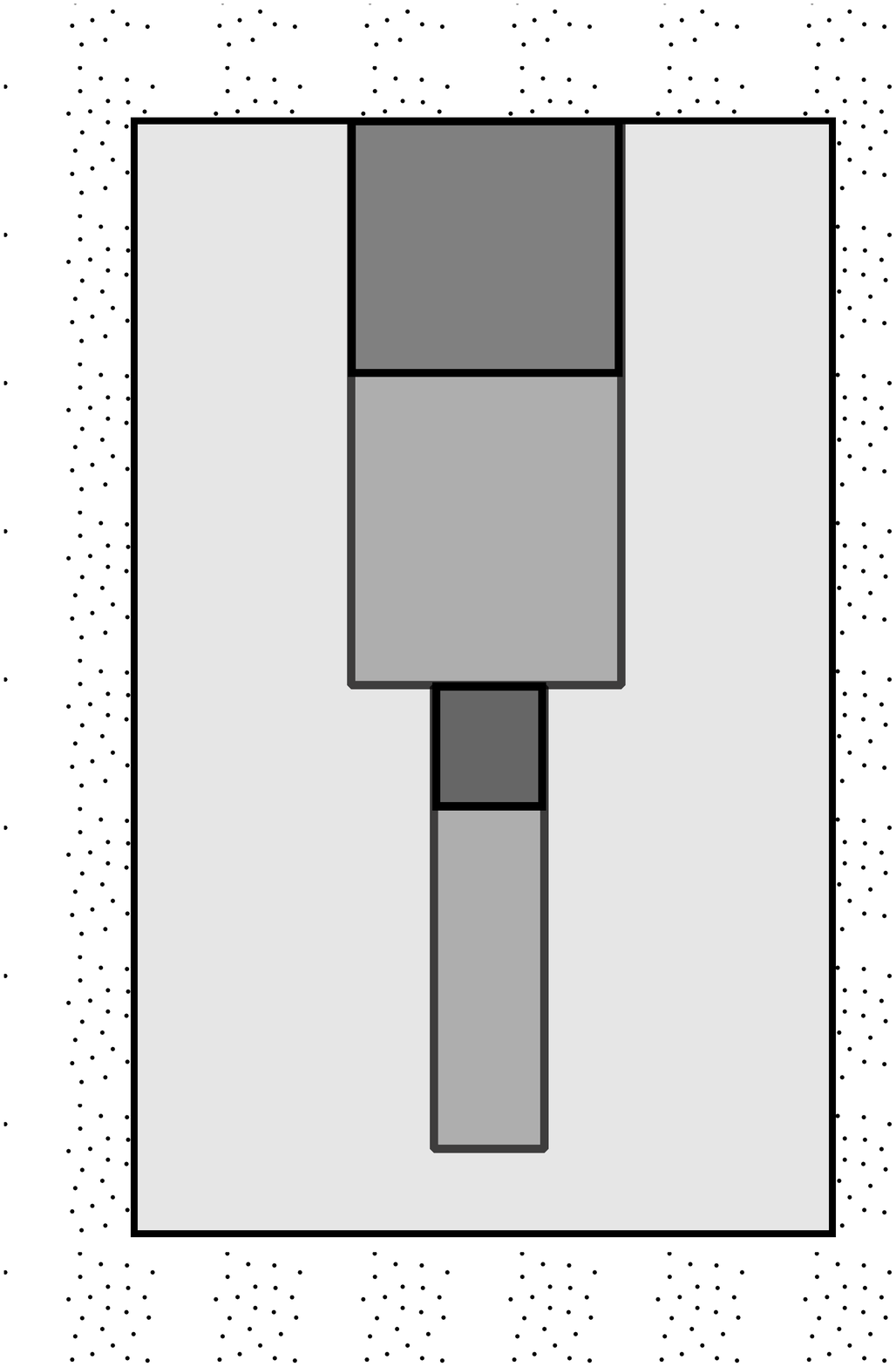}
\includegraphics[height=6cm]{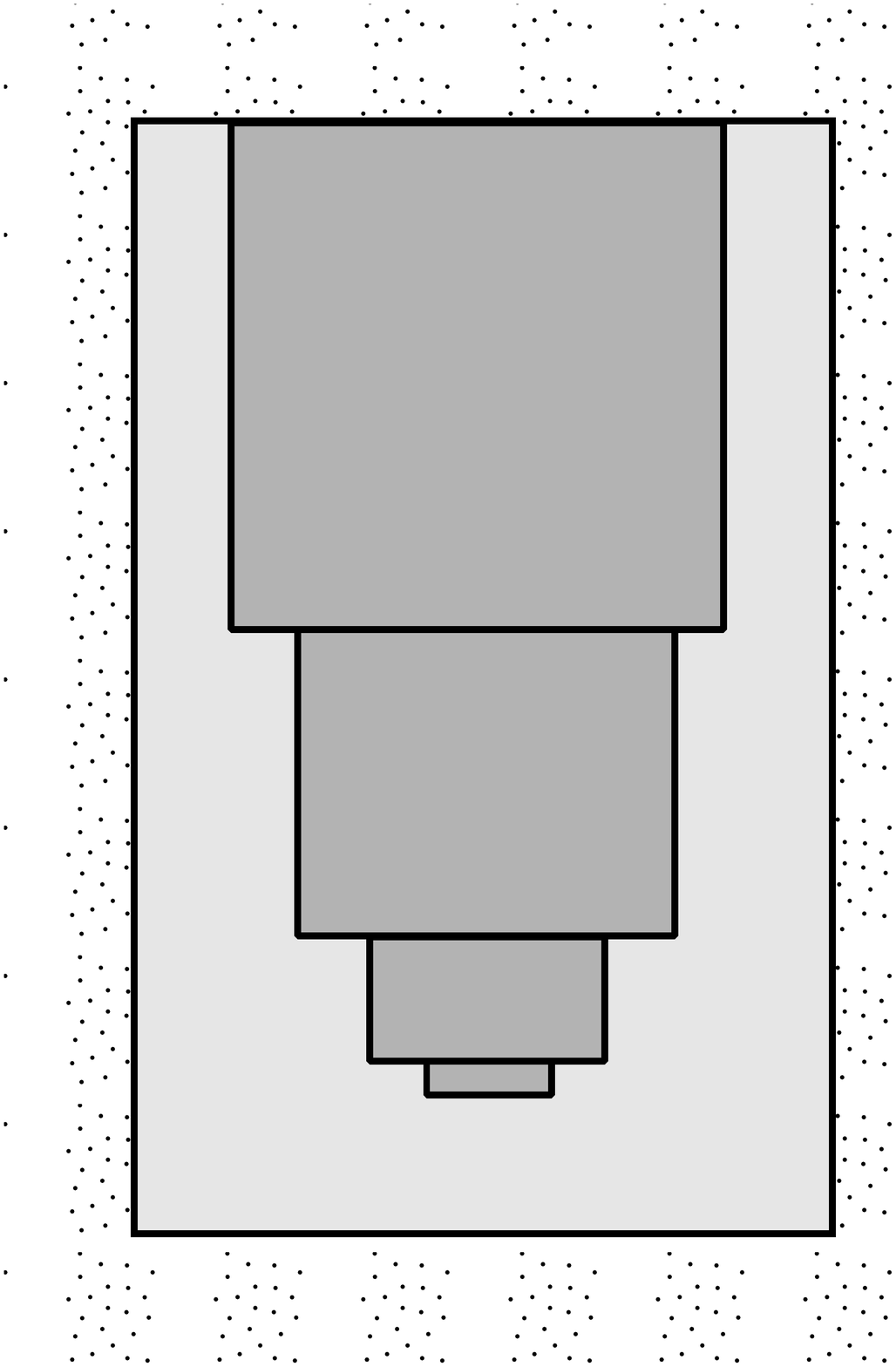}
\caption{The doubling property. On the left, the cylinders
  $Q_{\omega/4}$ and $Q_{\omega/2} (0,0,1)$. In the middle, the
  elongated cylinders $Q^0$ and $Q^1$.  On the right, the iterated
  cylinders $Q^0, \dots, Q^N$ (Lemma~\ref{lem:doubling}).}
\label{fig:doubling}
\end{figure}

If \(\inf_{Q_{\omega/4}}f=0,\)
there is nothing to prove. If not, the function
\[g = \frac{f}{\inf_{Q_{\omega/4}}f}\]
satisfies \eqref{eq:main} in $Q_2$ (up to translation in time -- this
is where we use that $s=0$) and
\[ | \{ g \ge 1 \} \cap Q_\omega| \ge |Q_{\omega/4}| = \delta_2 |Q_\omega|\]
for some universal $\delta_2$, where $Q_\omega$ plays the role of
$\hat Q$ in Lemma~\ref{lem:measure-to-pointwise}.  We then apply
Lemma~\ref{lem:measure-to-pointwise} (with time shifted by $+1$) to
$\tilde g = 1-g \le 1$, we get $g \ge \mathfrak h_0$ in
$B_{(\omega/2)^3} \times B_{\omega/2} \times (1-(\omega/2)^2,1]$, that
is to say, \eqref{eq:min-delay} indeed holds true.

Apply now the result to $\tilde f (x,v,t) = f (x,v,t-T)$ for
$T\in [0,1-\omega^2]$ and get
\begin{equation}\label{eq:min-trunc}
 \inf_{B_{(\omega/2)^3} \times B_{\omega/2} \times (0,1]} f \ge
 \mathfrak h_0 \inf_{Q^{0}} f.
\end{equation}
By applying \eqref{eq:min-trunc} on time intervals $(1,2]$, $(2,3]$
and $(3,4]$, we propagate the infimum till time $t=4$ and get the
desired result for $\mathfrak h = \mathfrak h_0^4$.
\end{proof}
Applying iteratively the previous lemma, we obtain straightforwardly
the following lemma whose proof is omitted.
\begin{lemma}[The iterated doubling property at the
  origin]\label{lem:doubling}
  There exists $\mathfrak h >0$ (universal) such that for any $f$
  non-negative super-solution of \eqref{eq:main} in
  $B_{2^{3N}} \times B_{2^N} \times (-1,T^{N})$, we have
\begin{equation}\label{eq:doubling}
 \inf_{Q^N} f \ge  \mathfrak h^N \inf_{Q^0} f
\end{equation}
with  
\[
Q^k = B_{R_k^3} \times B_{R_k} \times (T_{k-1},T_k] \quad \text{ for } \quad k \ge 1
\]
where $R_k =(\omega/4) 2^k$ and $T_k = \frac43 (4^{k}-1)$ for
$k \ge 0$.
\end{lemma}
\begin{remark}
  In \cite{im-book}, a measure estimate is also applied iteratively to
  prove a Harnack inequality for fully nonlinear parabolic equations
  in non-divergence form.
\end{remark}
We can now prove Proposition~\ref{prop:propagation}. 
\begin{proof}[Proof of Proposition~\ref{prop:propagation}]
In the following proof, we need iterated cylinders that are not centered
at the origin and with arbitrary radius. 
\[ Q^k_r(z):= \mathcal{T}_z \big( r Q^k \big). \]
The cylinder $Q^k$ is first scaled by $r$ (this is $rQ^k$) and then centered
around $z$ (this is $\mathcal{T}_z \big( r Q^k \big)$).

Let $z_\infty \in Q^+$ be such that $\min_{Q^+} f = f(z_\infty)$.
\begin{lemma}\label{lem:tech}
There exist $R$, $\Delta$, $r_0$ (small, universal)  such that
\begin{enumerate}
\item[a)] for all $r \in (0,r_0)$ and $z \in Q^-$, the iterated
  cylinders $Q^k_r(z)$ ($k \in \N$) which are included in
  $\{ t \le 0 \}$ are in fact included in $Q_1(0)$;
\item[b)]  the union of the iterated cylinders
  $\bigcup_{k=1}^{+\infty} Q^k_r(z)$ contains $Q^+$.
\end{enumerate}
\end{lemma}
The proof is elementary but tedious. It is given in Appendix. 

Applying Lemma~\ref{lem:doubling}, we get 
\[ \inf_{\Qel_r (z)} f \le \mathfrak h^{-N} \inf_{Q^N_r(z)} f \le
\mathfrak h^{-N} \min_{Q^+} f\]
with $N$ such that $z_\infty \in Q^N_r(z)$, i.e.
$r^{-1} (z^{-1} \circ z_\infty) \in Q^N$.  In particular,
$r^{-2}(t_\infty - t) \in [T^{N-1},T^N]$.  Since $z_\infty \in Q^+$
and $z \in Q^-$, we know that
\[ 4^{N-1} \le T^{N-1} \le 
\frac{t_\infty - t}{r^2} \le \frac{1/2+R^2}{r^2}.\]
In particular,
\[ \mathfrak h^{-N} \le \left( \frac{1/2+R^2}{4}  \right)^{\frac{q}2} r^{-q} \]
where $q = -\ln \gamma / \ln 2>0$. 
We get the desired inequality  with $C_{\text{pm}} = ((1/2+R^2)/4)^{\frac{q}2}$. 
The proof of the proposition is thus complete. 
\end{proof}


\subsection{Proof of the Harnack inequality}

We can now turn to the proof of Theorem~\ref{thm:harnack}.
\begin{proof}[Proof of Theorem~\ref{thm:harnack}]
  We first remark that replacing $f(x,v,t)$ with $f(x,v,t) + \|s\|_{L^\infty} t$ if
  necessary, we can assume that $s \ge 0$. Dividing $f$ by
  $2 \beta^{-1} \|s\|_{L^\infty}$ if necessary, we can assume that
  $\|s\|_{L^\infty} = \beta /2$ (if $s \not \equiv 0$). 

  We are going to find a universal constant $C=C_H$ such that
  \eqref{eq:harnack} cannot hold false. In other words, we
are going to find a universal $C_H$ such that 
\begin{equation}\label{eq:mM}
 m + 1 \le C_H M  
\end{equation}
entails a contradiction where
\[ M :=\sup_{Q^-} f= f (z_0) \quad \text{ and } \quad m :=\inf_{Q^+}
f= f (z_{\infty})\]
for some $z_0 \in Q^-$ and $z_\infty\in Q^+$. 
We used here the fact that $u$ is (H\"older) continuous.

Our goal is to construct by induction a sequence $(z_k)_{k \ge 0}$ in
$Q^-[1]$ (we recall that
$Q^- \subset Q^- [1] \subset Q^-[2] \subset Q_1$, see
Figure~\ref{cylinders}) such that
\begin{equation}\label{eq:fk-bis}
f (z_k) \ge (1 -\delta')^{-k} M
\end{equation}
for some universal $\delta' \in (0,1)$. This implies in particular
that $f(z_k) \to +\infty$ as $k \to +\infty$ which is absurd since $f$
is bounded in $Q^-$.

Remark first that \eqref{eq:fk-bis} holds true for $k=0$. Let us assume that 
we already constructed $z_0,\dots,z_k$ and let us construct $z_{k+1}$. 
Let $z_k = (x_k,v_k,t_k)$. We choose $r_k >0$ such that 
\begin{equation}\label{eq:rk-choice}
 f(z_k ) = r_k^{-2q} m 
\end{equation}
where $q$ is given by Proposition~\ref{prop:propagation}. Inequality
\eqref{eq:mM} and the induction hypothesis \eqref{eq:fk-bis} imply
\begin{equation}\label{eq:rk}
 r_k^{2q} \le C_H(1-\delta')^k . 
\end{equation}
\bigskip

From the decrease of oscillation  (Proposition~\ref{prop:osc-decrease-true}), we know 
that 
\[ 1 + \osc_{Q_{r_k}} f \ge (1-\delta)^{-1} \osc_{q_k} f \]
(recall $2 \beta^{-1}\|s\|_{L^\infty}=1$)
with 
\[
Q_k = Q_{r_k}(z_k) \quad \text{ and } \quad q_k = Q_{\omega r_k/4} (z_k).
\]
 In particular, $z_k \in q_k$.  Let $z_{k+1} \in Q_k$ be such
that
\[ \max_{Q_k} f = f(z_{k+1}) .\]
Then we get 
\begin{equation}\label{eq:fk+1}
1 + f(z_{k+1}) \ge (1-\delta)^{-1} \left(f(z_k) - \min_{q_k} f \right). 
\end{equation} 

Recall that $z_k \in Q^-[1]$.  Choosing $C_H$ small, we can ensure
through \eqref{eq:rk} that $Q_{r_k}(z_k) \subset Q^-[2]$.  We
also remark that
\[ q_k \supset \Qel_{(\omega/4)^2 r_k} (z_k).\]
We thus can apply Proposition~\ref{prop:propagation}  and get
\[ \min_{q_k} f \le \min_{\Qel_{(\omega/4)^2 r_k} (z_k)} f  \le \tilde C_{\text{pm}}  r_k^{-q} m\]
with $\tilde C_{\text{pm}} = C_{\text{pm}} (4 /\omega)^q$. 
The use of \eqref{eq:rk-choice} in the previous inequality yields
\begin{equation}\label{eq:mink}
 \min_{q_k} f \le \tilde C_{\text{pm}} r_k^q f(z_k) \le \tilde 
C_{\text{pm}} \sqrt{C_H} f(z_k).
\end{equation}
Now combining \eqref{eq:fk+1} and \eqref{eq:mink}, we get
\[ 1+ f(z_{k+1}) \ge (1-\delta)^{-1} (1- \tilde C_{\text{pm}} \sqrt{C_H}) f(z_k). \]
Use next that $1 \le C_H M$ (this is a consequence of \eqref{eq:mM}) 
and the induction hypothesis and get
\begin{align*}
 f(z_{k+1}) & \ge (1-\delta)^{-1} (1- \tilde C_{\text{pm}} \sqrt{C_H}) (1-\delta')^{-k} M - C_H M\\
& \ge \mathfrak j (1-\delta')^{-k} M. 
\end{align*}
with
\[\mathfrak j = (1-\delta)^{-1} (1- \tilde C_{\text{pm}} \sqrt{C_H}) - C_H.\]
We thus choose $\delta'$ such that 
\[ (1-\delta')^{-1} = \mathfrak j \]
and we can choose $C_H$ small enough so that $\delta' \in (0,1)$. 
In particular we get 
\[ f(z_{k+1}) \ge (1-\delta')^{-k-1} M\]
which is the desired inequality.
\bigskip

We are left with proving that the sequence $\{ z_k \}$ 
stays in $Q^-[1]$. The fact that $z_{k+1}$ lies in $Q_{r_k} (z_k) = \mathcal{T}_{z_k} (Q_{r_k}  (0))$. 
This implies in particular that $|v_{k+1}-v_k|\le r_k$ which in turn yields
\[ 
|v_k -v_0| \le \sum_{l\ge 0} r_l \le C_H^{1/(2q)} \sum_{l \ge 0} (1-\delta')^{\frac{k}{2q}} 
= \frac{C_H^{1/(2q)}}{1-(1-\delta')^{1/(2q)}}.
 \]
 Using now that the fact that $\delta'$ is explicitely given as a
 function of $\delta$ and $C_H$ (see above), we conclude that
 $|v_k-v_0|$ can be arbitrarily small uniformly in $k$. We
 can argue in the same spirit for $|x_k-x_0|$ and $|t_k-t_0|$. Since
 $z_0 \in Q^-$, we conclude that we can indeed ensure that $z_k$ lies
 in $Q^-[1]$. The proof of the theorem is now complete.
\end{proof}

\section{Local gain of regularity for sub-solutions}
\label{sec:local-gain-regul}

In this section, we investigate the regularity of \emph{sub-solutions}
to Eq.~\eqref{eq:main} beyond the gain of integrability proved
above. Observe that, on the one hand, Theorem~\ref{thm:gain} applies
to sub-solutions but only concludes to the gain of integrability. On
the other hand, Theorem~\ref{thm:gain-diff} proves a gain of Sobolev
regularity but only applies to \emph{solutions} (not
sub-solutions). It might seem, at first sight, that the lack of
ellipticity in all directions means the gain of regularity of
solutions is false, since in the elliptic and parabolic case it is
entirely based on the energy estimate. However we show here that,
using the local upper bound proved above by the De~Giorgi--Moser
iteration, and refined averaging lemmas, this result still holds in
essence for our equation, even though the gain of regularity is only
$H^s$ with $s>0$ small. We prove the following result:
\begin{theorem}[Gain of regularity for non-negative
  sub-solutions]\label{thm:gain-diff-sub}
  Consider $z_0 \in \R^{2d+1}$ and two cylinders
  $\Qint :=Q_{r_1}(z_0)$ and $\Qext:=Q_{r_0}(z_0)$ with
  $0< r_1 < r_0$.  Then there is some $\mathfrak s \in (0,1/3)$ so that any
  weak non-negative sub-solution $f$ of \eqref{eq:main} in $\Qext$
  satisfies
\begin{equation}
  \label{eq:gain-diff-sub}
  \n{f}_{H^{\mathfrak s}_{x,v,t} (\Qint)} \le C 
\left(  \n{f}_{L^2(\Qext)} + \n{s}_{L^\infty(\Qext)}  \right)
\end{equation}
with $C =C(d,\lambda,\Lambda,\Qext,\Qint)$. 
\end{theorem}
\begin{proof}[Proof of Theorem~\ref{thm:gain-diff-sub}]
  We define $\Q12$ in between $\Qint$ and $\Qext$ and the same
  truncation functions as before. Theorem~\ref{thm:sup-sub} implies that 
  \begin{equation*}
    \n{f}_{L^\infty (\Q12)} \lesssim \n{f}_{L^2(\Qext)} + \n{s}_{L^\infty(\Qext)}.
  \end{equation*}

We want to apply \cite[Theorem~1.3]{bouchut} on $f$ in $\Q12$. However
since $f$ is only a sub-solution it satisfies the equation 
\begin{equation*}
  \partial_t f + v \cdot \nabla_x f = \nabla_v \cdot \left(A \nabla_v
    f  \right) + B \cdot \nabla_v f + s - \mu \qquad \text{ in } \Qext
\end{equation*}
where we have included the defect non-negative measure $\mu \ge 0$
accounting for the inequation. We can now repeat the reasoning from 
the proof of Lemma~\ref{lem:isoperim} and reduce to
the case
\begin{equation*}
  \partial_t g + v \cdot \nabla_x g = \nabla_v \cdot \left(A \nabla_v
    g  \right) + \nabla_v \cdot H_1 + H_0 - \tilde \mu \qquad \text{ in } \R^{2d+1}
\end{equation*}
with $g \equiv f$ in $\Qint$ and $g$, the measure $\tilde \mu \ge 0$, $H_0$ and
$H_1$ supported in $\Q12$, and with $g$, $\nabla_v g$, $H_0$ and $H_1$
bounded in $L^2$ on $\Q12$. Then by integrating in $x,v,t$ we deduce
that $\tilde \mu$ has bounded variation in terms of the previous
bounds.  Since for $q > (4d+2)$, the space $W^{\frac12,q}_{x,v,t}$
embeds into continuous bounded functions of $x,v,t$, we deduce that
the space of measures is included in $W^{-\frac12,q^*}_{x,v,t}$ and
therefore
\begin{equation}\label{eq:mug}
  \tilde \mu = \left( 1- \Delta_{x,t} \right)^{\frac14} \left( 1 -
    \Delta_v \right) h \quad \mbox{
    with } \quad h \in L^{q^*}(\Q12)
\end{equation}
and the bound on the $L^{q^*}(\Q12)$ depends on the previous bounds
above, and where $q^* = 1/(1-1/q)$ is the conjugate exponent of
$q$. Observe that $q^*$ is striclty smaller than $2$ and close to one,
for instance $q^* \in (1,14/13)$ in dimension $d=3$. We then apply
\cite[Theorem~1.3]{bouchut} with $\kappa =1$, $r=\frac12$, $m=2$,
$\beta=1$, $p=q^*$: we deduce that $g$ belongs to
$W^{\frac18,p}_{x,t}L^p_v$ (observe that we use a full Laplacian
derivative in $v$ in Eq.~\eqref{eq:mug} in order to be in the
framework of \cite[Theorem~1.3]{bouchut}, even though
$(1 - \Delta_v )^{1/4}$
would have been enough for the purpose of having $h \in L^{q^*}$). By
interpolation with the $L^\infty$ estimate, we obtain then that
$g \in H^{\mathfrak s}_{x,t} L^2_v$ for some
$\mathfrak s \in (0,\frac18)$ small enough. Finally, we combine the
latter estimate with the energy estimate $g \in L^2_{x,t} H^1_v$ we
conclude with $g \in H^{\mathfrak s}_{x,v,t}$. Since the truncation
function is equal to one on the smaller cube $\Qint$, it translates
into $f \in H^{\mathfrak s}_{x,v,t}$ on $\Qint$ and concludes the
proof.
\end{proof}

\section{Gain of integrability of the velocity gradient}
\label{sec:l2+eps}

This section is devoted to the proof of the following theorem.
\begin{theorem}[Gain of integrability for $\nabla_v f$]\label{thm:l2+eps}
  Let $f$ be a solution of \eqref{eq:main} without lower order
terms
  ($B \equiv 0$ and $s\equiv 0$) in some cylinder $Q_{r_0} (z_0)$. There exists a
  universal $\eps >0$ such that for all $Q[i] = Q_{r_i} (z_0)$,
  $i=0,1,2$ with $r_2<r_1 < r_0$, $\nabla_v f \in L^{2+\eps}(Q_2)$
\begin{equation}\label{eq:l2+eps}
\int_{Q[2]} |\nabla_v f|^{2+\eps} \dd z \le 
C \left( \int_{Q[1]} |\nabla_v f |^2 \dd z \right)^{\frac{2+\eps}2}
\end{equation}
with $C = C(d,\lambda,\Lambda,Q_2,\Qint,\Qext)$. 
\end{theorem}
The proof follows along the lines of the one of
\cite[Theorem~2.1]{gs}. It consists in deriving an almost reverse
H\"older inequality which in turn implies the result thanks to the
analogous of \cite[Proposition~1.3]{gs}. The following
measure-theoretical lemma will be used as a black box in the proof of
Theorem~\ref{thm:l2+eps}. It implies the use of cylinders with different
shape: 
\[\cQ (z_0,r) = \{ z = (x,v,t) : |x_i -x_i^0| < r^3, |v_i-v_i^0| < r, -r^2 < t-t_0 \le 0\}\]
where $x=(x_1,\dots,x_d)$ and $v=(v_1,\dots,v_d)$. 
The scaling of the equation preserves this family of cylinders but not the
Lie group action $\mathcal{T}_z$. 
\begin{lemma}[A Gehring lemma]\label{lem:gehring}
  Let $g \ge 0$ in $\cQ$ such that there exists $q >1$ such that for all
  $z_0 \in \cQ$ and $R$ such that $\cQ_{4R} (z_0) \subset Q$, 
\[ \fint_{\cQ_R(z_0)} g^q \dd z \le b \left(\fint_{\cQ_{4R}(z_0)} g \dd z
\right)^q + \theta \fint_{\cQ_{4R}(z_0)} g^q \dd z \]
for some $\theta >0$. There exists $\theta_0 = \theta_0 (q,d)$ such that 
if $\theta < \theta_0$, then $g \in L^p_{\mathrm{loc}} (Q)$ for $p \in [q,q+\eps)$ and 
\[ \left( \fint_{\cQ_R} g^p \dd z \right)^{\frac1p} \le c_p \left(
  \fint_{\cQ_{4R}} g^q \dd z \right)^{\frac1q}, \]
the constants $\eps>0$ depends only on $b,q,\theta$ and dimension, and
$c_p$ further depends on $p$. 
\end{lemma}
The proof of Lemma~\ref{lem:gehring} is an easy adaptation of the one
of \cite[Proposition~5.1]{gm}, by changing Euclidian cubes with
cylinders $\cQ_R$. \bigskip

The proof of Theorem~\ref{thm:l2+eps} is a consequence of some
estimates involving weighted means of the solution. Given $z_0 \in
\R^{2d+1}$, they are defined as follows
\[ 
\tilde{f}_{2R} (t) = \frac1{cR^{4d}} \int_{\R^{2d}} f(t,x,v) \chi_{2R} (x,v,t)
\dd x \dd v   
\]
(for some $c$ defined below) where $\chi_{2R}$ is a cut-off function such that 
\[ \chi_{2R} (x,v,t) = \prod_{i=1}^d\phi_{R^3}(x_i-x_i^0) \phi_R(v_i-v_i^0) \] with $\phi_R (a) =
\phi(a/R)$ for some $\phi$ such that $\sqrt{\phi} \in C^\infty(\R)$
and $\phi \equiv 1$ in $[-1,1]$ and $\supp \phi\subset [-2,2]$.   
We remark  that
$\chi_{2R} \equiv 1$ in $\cQ_R$ and $\chi_{2R} \equiv 0$ outside $\cQ_{2R}$. 
\begin{lemma} \label{lem:caccio-bis} Let $f$ be a solution
  of \eqref{eq:main} in $\cQ_0$. Then for $\cQ_{3R} (z_0) \subset
  \cQ_0$,
\begin{eqnarray}
\label{eq:caccio-mean}
 \int_{\cQ_R(z_0)} |\nabla_v f |^2 \dd z \le C R^{-2} \int_{\cQ_{2R}(z_0)}
  |f - \tilde{f}_{2R}|^2 \dd z \\
\label{eq:poincare}
 \sup_{t \in (t_0-R^2,t_0]} \int_{\cQ_R^t(z_0)} |f(t) - \tilde{f}_R
  (t)|^2 \dd x \dd v \le C  \int_{\cQ_{3R}(z_0)} |\nabla_v f|^2 \dd z
\end{eqnarray}
where $\cQ_R^t (z_0)=  \{ (x,v): (t,x,v) \in \cQ_R(z_0)\}$.
\end{lemma}
\begin{remark}
This lemma corresponds to \cite[Lemmas~2.1 \& 2.2]{gs}. 
\end{remark}
\begin{proof}
For the sake of clarity, we put $z_0=0$ and $R=1$.  
  Consider $\tau_{2} \in C^\infty (\R,\R)$ such that $0 \le \tau_{2} \le 1$,
  $\tau_{2} \equiv 0$ in $(-\infty,-2^2]$ and $\tau_{2} \equiv 1$ in
  $[-1,0]$.  Use $2(f-\tilde{f}_{2}) \chi_{2}  \tau_{2} $ as a test
  function for \eqref{eq:main} and get 
\begin{multline*}
\int_{\R^{2d}}  (f(0)-\tilde{f}_{2}(0))^2 \chi_{2} \dd x \dd v + 2 \int_{\R^{2d+1}} (A \nabla_v
f \cdot \nabla_v f) \chi_{2} \tau_{2} \dd x \dd v \dd t \\
= \int_{\R^{2d+1}}(f-\tilde{f}_{2})^2 \chi_{2} (\partial_t \tau_{2} ) \dd x \dd v \dd t
-   \int_{\R^{2d+1}} v \cdot \nabla_x \left[ (f - \tilde{f}_{2} )^2 \right]
\chi_{2} \tau_{2} \dd x \dd v \dd t \\
- 2 \int_{\R^{2d+1}} (f-\tilde{f}_{2}) A \nabla_v f \cdot \nabla_v \chi_{2} \tau_{2} \dd x \dd v \dd t.
\end{multline*}
Remark that the definition of $\tilde{f}_{2}$ implies that the remaining
term \[- 2 \int_{\R^{2d+1}} (\partial_t \tilde{f}_{2}) (f-\tilde{f}_{2}) \chi_{2} \tau_{2}\]
vanishes. This equality yields 
\begin{align*}
\int_{\R^{2d}}  (f(0)-\tilde{f}_{2}(0))^2 \chi_{2} \dd x \dd v + \lambda \int_{\R^{2d+1}}
  |\nabla_v f|^2 \chi_{2} \tau_{2} \dd x \dd v \dd t \\
\le \int_{\R^{2d+1}} (f-\tilde{f}_{2})^2 \left(\chi_{2} |\partial_t \tau_{2} | 
+ |v \cdot \nabla_x \chi_{2}| \tau_{2}  + \frac{\Lambda^2}{\lambda} |
  \nabla_v \sqrt{\chi_{2}} |^2 \tau_{2} \right) \dd x \dd v \dd t
\end{align*}
which yields \eqref{eq:caccio-mean}. Changing the final time, we also get
\[
\sup_{t \in (-1,0]} \int_{\R^{2d}}  \big[ f(t)-\tilde{f}_{2}(t) \big]^2 \chi_{2} (t) \dd x
\dd v 
\le C   \int_{\cQ_{2}} |f - \tilde{f}_{2}|^2 \dd x \dd v \dd t.
\]
Now the function $F= f -\tilde{f}_{2}$ is such that $\int F(x,v,t) \dd
x \dd v =0$. 
In particular, we have 
\[ 
\int_{\cQ_{2}} (f-\tilde{f}_{2})^2 \dd x \dd v \dd t \le C \int_{\cQ_{2}} 
\left(|\nabla_v f|^2 +  |D^{\frac13}_x f|^2\right) \dd x \dd v \dd t. 
\] 

Observe that if there are no lower order terms ($B=0$ and $s=0$), then
we have for all $q \in (1,2]$,
\begin{equation}
\label{eq:gain-x-sol-q}
 \int_{\cQ_{2}}  |D^{\frac13}_x f|^q \dd x \dd v \dd t \le C \int_{\cQ_{3}} 
 |\nabla_v f|^q  \dd x \dd v \dd t.
\end{equation}
Indeed, in view of the proof of
\eqref{eq:gain-diff}, it is enough to apply
\cite[Theorem~1.3]{bouchut} with such a $q$ and use the 
Poincar\'e inequality (assuming the cutoff functions to have convex
super-level sets). 

Combining the three previous estimates yields 
\[ 
\sup_{t \in ( -1,0]} \int_{Q^t_1}  (f(t)-\tilde{f}_{2}(t))^2 \chi_{2}(t) \dd x
\dd v 
\le C  \int_{\cQ_{3}}  |\nabla_v f|^2  \dd x \dd v \dd t. 
\]
Finally, we write for $t \in (-1,0]$ 
\begin{align*}
\frac12 \int_{\cQ_1^t}  (f(t)-\tilde{f}_{1}(t))^2 \chi_{2}(t) \le &  \int_{\cQ_1^t}  (f(t)-\tilde{f}_{2}(t))^2 \chi_{2}(t)
 + \int_{\cQ_1^t}  (\tilde{f}_{2}(t)-\tilde{f}_{R}(t))^2 \chi_{2}(t)\\
\le &  \int_{\cQ_1^t}  (f(t)-\tilde{f}_{2}(t))^2 \chi_{2}(t) + |\cQ_1^t| \left(
      \frac1c \int_{\cQ_1^t} (f- \tilde f_{2} (t)) \chi_{1} (x,v,t) \dd x \dd v \right)^2 \\
 \le & \ C \int_{\cQ_1^t}  (f(t)-\tilde{f}_{2}(t))^2 \chi_{2}(t) 
\end{align*}
and we get the second desired estimate since $\chi_{2} \equiv 1$ in $\cQ_1$. 
\end{proof}
We now turn to the proof of Theorem~\ref{thm:l2+eps}. The use of
\eqref{eq:gain-x-sol-q} is the main difference with \cite{gs}.
\begin{proof}[Proof of Theorem~\ref{thm:l2+eps}]
  Pick $p>2$ and let $q$ denote its conjugate exponent: $\frac1q +
  \frac1{p} =1$. We follow \cite{gs} in writing (omitting the center of cylinders $z_0$),
thanks to \eqref{eq:caccio-mean},  
\begin{align*}
  \fint_{\cQ_{1}} |\nabla_v f |^2    & \lesssim \int_{\cQ_2} |f - \tilde{f}_2|^2 \,   \\
& \le \sup_{t \in (t_0-4,t_0]} \left(\int_{Q^t_2} |f - \tilde{f}_2|^2 \right)^{\frac12}  
\int_{t_0-4}^{t_0} dt \left( \int_{Q^t_2} |f - \tilde{f}_2|^2\right)^{\frac12} \\
& \lesssim  \left( \int_{\cQ_{4}} |\nabla_v f|^2 \,  \right)^{\frac12} 
 \int_{t_0-4}^{t_0} dt  \left( \int_{Q^t_2} |f - \tilde{f}_2|^q\right)^{\frac1{2q}} 
\left( \int_{Q^t_2} |f - \tilde{f}_2|^{p}\right)^{\frac1{2p}}
\end{align*}
where   \eqref{eq:poincare} and H\"older inequality are used successively. 

We now use Sobolev inequalities and H\"older inequality (twice)
successively to get
\begin{align*}
  \fint_{\cQ_{1}} |\nabla_v f |^2   \,  \lesssim & 
 \left( \int_{\cQ_{4 }} |\nabla_v f|^2 \,  \right)^{\frac12} 
 \times  \left[ \int_{t_0-4}^{t_0}  \left( \int_{Q^t_2} |\nabla_v f|^q +  |D_x^{1/3} f|^q \right)^{\frac1{2q}}
  \dd t \right] \\
& \hspace{3cm}\times \left( \int_{Q^t_2}   |\nabla_v f |^2 +  |D_x^{1/3} f|^2 \right)^{\frac1{4}} \\
 \lesssim &  \left( \int_{\cQ_{4 }} |\nabla_v f|^2 \,  \right)^{\frac12} \left( \int_{\cQ_2}  |\nabla_v f|^q +  |D_x^{1/3} f|^q \right)^{\frac1{2q}}  \\
& \hspace{2cm} \times  \left( \int_{t_0-4}^{t_0} \left(\int_{Q^t_2}   |\nabla_v f
  |^2 +  |D_x^{1/3} f|^2 \right)^{\frac{q}{2(2q-1)}} \dd t \right)^{\frac{2q-1}{2q}} \\
 \lesssim &  \left( \int_{\cQ_{4}} |\nabla_v f|^2 \,  \right)^{\frac12} 
 \times \left( \int_{\cQ_2}  |\nabla_v f|^q +  |D_x^{1/3} f|^q \right)^{\frac1{2q}}  \times  \left( \int_{\cQ_2}   |\nabla_v f |^2 +  |D_x^{1/3} f|^2  \right)^{\frac14} .
\end{align*}
We now use \eqref{eq:gain-x-sol-q} and get 
\begin{align*}
  \fint_{\cQ_{1}} |\nabla_v f |^2    \,  &\lesssim 
\left( \int_{\cQ_{4}} |\nabla_v f|^2 \,  \right)^{\frac12} 
 \left( \int_{\cQ_2}  |\nabla_v f|^q \right)^{\frac1{2q}} 
 \left( \int_{\cQ_2}   |\nabla_v f |^2 \right)^{\frac14}  \\
& \lesssim  \left( \int_{\cQ_{4}} |\nabla_v f|^2 \,  \right)^{\frac34}  \left( \int_{\cQ_2}  |\nabla_v f|^q \right)^{\frac1{2q}}.
\end{align*}
Now use and get for all $\eps>0$,
\begin{align*}
  \fint_{\cQ_{1}} |\nabla_v f |^2   & \lesssim  
  \left( \fint_{\cQ_{4}} |\nabla_v f|^2 \, 
  \right)^{\frac34} \left( \fint_{\cQ_{4}} |\nabla_v f|^q 
  \right)^{\frac1{2q}} \\
& \lesssim   \left( \fint_{\cQ_{4}} |\nabla_v f|^2  
  \right)^{\frac34} \left( \fint_{\cQ_{4}} |\nabla_v f|^q 
  \right)^{\frac1{2q}} .
\end{align*}
After rescaling, we get the following
\begin{align*}
  \fint_{\cQ_{R}} |\nabla_v f |^2  &\lesssim   \left( \fint_{\cQ_{4R}} |\nabla_v f|^2  
  \right)^{\frac34} \left( \fint_{\cQ_{4R}} |\nabla_v f|^q 
  \right)^{\frac1{2q}} \\
& \lesssim \eps \fint_{\cQ_{4R}} |\nabla_v f|^2 + c_{\eps}  \left( \fint_{\cQ_{4R}} |\nabla_v f|^q 
  \right)^{\frac2{q}}.
\end{align*}
Apply now Proposition~\ref{lem:gehring} in order to achieve the proof
of Theorem~\ref{thm:l2+eps}.
\end{proof}

\appendix

\section{Known estimates for the Landau equation}
\label{sec:landau}

\begin{lemma}[Lower bound - \cite{dv,luislandau}]\label{lem:upperbound}
  Assume there exist positive constants $M_1,M_0,E_0$ and $H_0$ such
  that \eqref{e:meh} holds true. Then
\[ \det A [f] \ge c (1+|v|)^\kappa \]
with 
\[ 
\kappa =
\begin{cases} 
(d-1)(\gamma+2) + \gamma & \text{ if } \gamma \in [-2,0] \\
3 \gamma +2 & \text{ if } \gamma \in [-d,-2)
\end{cases}
\]
where $c$ only depends on dimension, $\gamma$, $M_0$, $M_1$, $E_0$ and $H_0$.
\end{lemma}
\begin{lemma}[Upper bounds - \cite{dv,luislandau}]\label{lem:lowerbound}
 Assume there exist positive constants $M_1,M_0,E_0$ and $H_0$ such
  that \eqref{e:meh} holds true. 
Assume that $f \in  L^\infty (\R^d)$. Then
\begin{align*}
 |A[f]| &\le \begin{cases} 
C (1+|v|)^{\gamma+2} & \text{ if } \gamma \in [-2,0]  \\
C \|f\|_{\infty}^{\frac{|\gamma+2|}d} & \text{ if } \gamma \in [-d,-2)   
\end{cases}
\\
|B[f]| &\le  \begin{cases} 
C (1+|v|)^{\gamma+1} & \text{ if } \gamma \in [-1,0]  \\
C \|f\|_{\infty}^{\frac{|\gamma+1|}d} & \text{ if } \gamma \in [-d,-1)   
\end{cases} \\
|c[f]| & \le 
\begin{cases}
C & \text{ if } \gamma =0 \\
C \|f\|_{\infty}^{\frac{|\gamma|}d} & \text{ if } \gamma \in [-d,0).   
\end{cases}
\end{align*}
where $C$ only depends on dimension, $\gamma$, $M_0$, $E_0$.
\end{lemma}

\section{Proof of a technical lemma}
\label{sec:technical}

\begin{proof}[Proof of Lemma~\ref{lem:tech}]
To justify a) and b), we remark that 
\[ 
\mathcal{P}^- \subset \bigcup_{k=1}^{+\infty}  Q^k \subset \mathcal{P}^+
\] 
where
\begin{align*}
\mathcal{P}^- & := \{ (y,w,s): s \ge \frac43 \left(\frac{4^2}{\omega^2} \rho^2 -1\right), 
|y|\le \rho^3, |w|\le
\rho \}, \\
\mathcal{P}^+ & := \{ (y,w,s): s \ge \frac43 \left(\frac{4}{\omega^2} \rho^2 -1\right), |y|\le \rho^3, |w|\le
\rho \}, 
\end{align*}
see Figure~\ref{fig:paraboloids}.
\begin{figure}[ht]
\includegraphics[height=5cm]{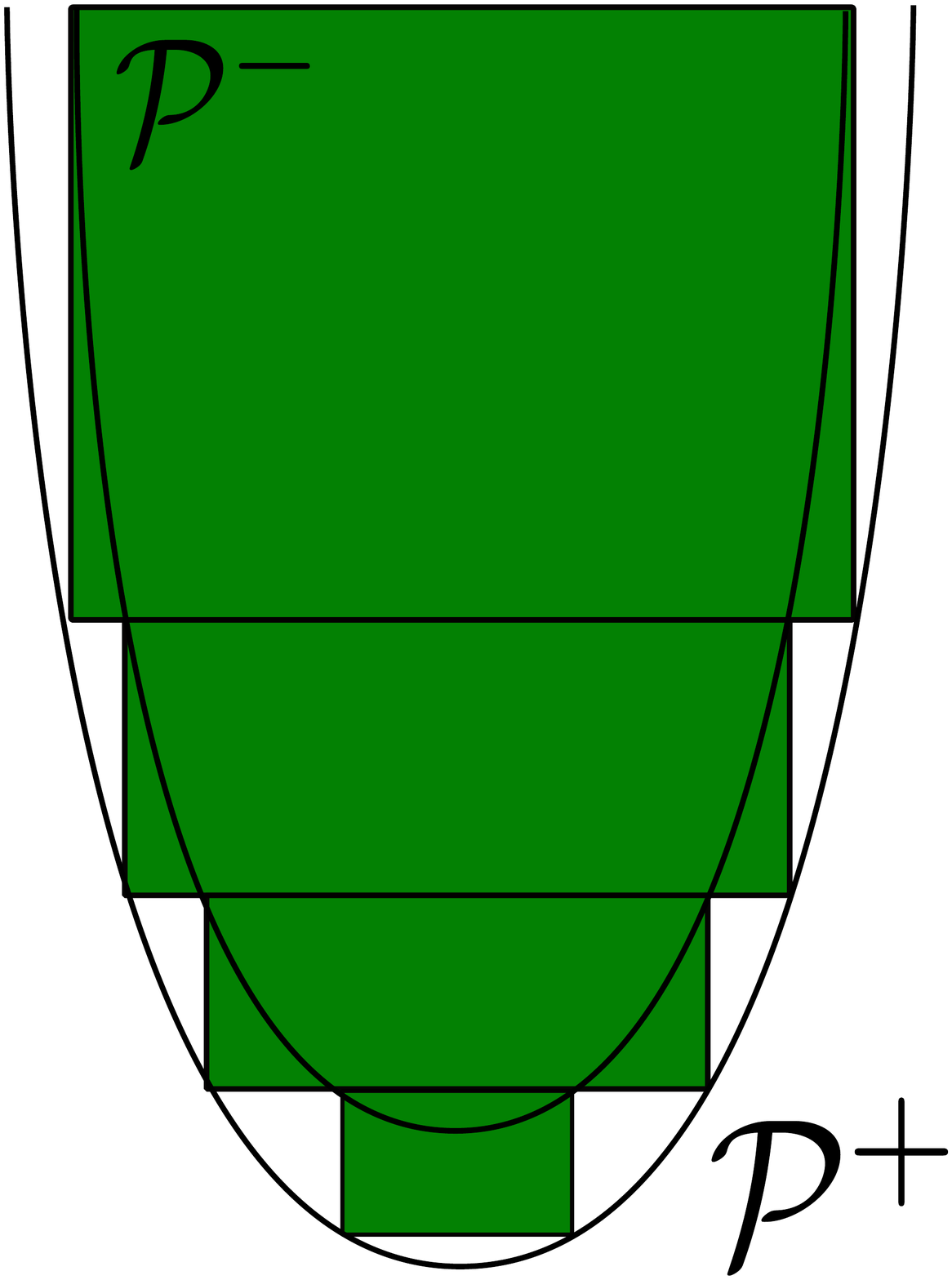}
\caption{Paraboloids containing/contained in the union of iterated cylinders.}
\label{fig:paraboloids}
\end{figure}

In what follows, $R$ and $r_0$ are chosen as functions of $\Delta$. 
In particular, 
\[ R \le \sqrt{\Delta} \quad \text{ and } \quad r_0 \le \sqrt{\Delta}. \] 

As far as a) is concerned, we should ensure that for all $z \in Q^-$ and $r \in (0,r_0)$,
\[ (z \circ r \mathcal{P}^+) \cap \{ t \le 0 \} \subset Q_1(0) .\]
If $z=(x^-,v^-,t^-)$ and $z^+=(x^+,v^+, t^+) \in r\mathcal{P}^+$ are such that 
\( z \circ z^+ \in \{ t \le 0 \}\), we have 
\begin{align*}
0 & \ge t^- +  t^+ \\
& \ge (-\Delta -R^2) +  \frac43((4/\omega^2) \rho^2 -r^2) \\
& \ge - 4 \Delta +  (4^2/3\omega^2) \rho^2
\end{align*}
where $\rho = |v^+|$. 
This implies in particular
\[  \rho^2 \le \frac{3\omega^2}4  \Delta .\]
In particular, for $\Delta \in (0,1)$, 
\begin{align*}
 |v^- + v^+ | & \le R + \rho \\
&\le  (1 + \sqrt{3}\omega/2) \sqrt{\Delta}\\
|x^- + x^+ + t^+ v^- | & \le R^3 + \rho^3 + R \\
& \le (1 + (\sqrt{3}\omega/2)^3 ) \Delta^{3/2} + \sqrt{\Delta} \\
& \le (2 + (\sqrt{3} \omega/2)^3 ) \sqrt{\Delta}.
\end{align*}
We thus can choose $\Delta$ small enough (recall $\omega=1/\sqrt{5}$) to ensure a).  

As far as b) is concerned, 
notice that for $z^+ \in Q^+$ and $z \in Q^-$, we have 
\[ z^{-1} \circ z^+ = (t^+-t,x^+-x-(t^+-t)v,v^+-v) .\]
Choosing $R^2 \le \Delta \le \frac12$ we have $2R \le (4R)^{\frac13}$ and we get
\begin{align*}
&  |v^+-v| \le 2R  \le (4 R)^{\frac13}\\
& |x^+-x - (t^+-t) v | \le 
2R^3 + (\Delta + R^2) R = 3R^3 + \Delta R \le 4 R
\end{align*}
(since $R\le 1$ and $\Delta \le 1$) and 
\[ t^+ -t \ge \Delta -R^2. \]
In particular $z^{-1} \circ z^+ \in r\mathcal{P}^-$ if 
\[
\Delta  - R^2  \ge \frac43 \left( \frac{4^2}{\omega^2} (4R)^{\frac13} - r^2 \right). 
\]
It is enough satisfy
\[
 \Delta \ge R^2 +\frac{4^3}{3\omega^2} (4R)^{\frac13}. 
\]
Hence, for $\Delta$ given, we can choose $R = R(\Delta)$ small
enough to get the desired inequality and in turn point b). 
\end{proof}

\bibliographystyle{plain}
\bibliography{ndgh}

\def\cprime{$'$} \def\cprime{$'$}
\begin{thebibliography}{10}

\bibitem{Agosh}
Valeri~I. Agoshkov.
\newblock Spaces of functions with differential-difference characteristics and
  the smoothness of solutions of the transport equation.
\newblock {\em Dokl. Akad. Nauk SSSR}, 276(6):1289--1293, 1984.

\bibitem{MR2556715}
Radjesvarane Alexandre.
\newblock A review of {B}oltzmann equation with singular kernels.
\newblock {\em Kinet. Relat. Models}, 2(4):551--646, 2009.

\bibitem{MR3375485}
Radjesvarane Alexandre, Jie Liao, and Chunjin Lin.
\newblock Some a priori estimates for the homogeneous {L}andau equation with
  soft potentials.
\newblock {\em Kinet. Relat. Models}, 8(4):617--650, 2015.

\bibitem{MR2679369}
Radjesvarane Alexandre, Yoshinori Morimoto, Seiji Ukai, Chao-Jiang Xu, and Tong
  Yang.
\newblock Regularizing effect and local existence for the non-cutoff
  {B}oltzmann equation.
\newblock {\em Arch. Ration. Mech. Anal.}, 198(1):39--123, 2010.

\bibitem{MR2795331}
Radjesvarane Alexandre, Yoshinori Morimoto, Seiji Ukai, Chao-Jiang Xu, and Tong
  Yang.
\newblock Global existence and full regularity of the {B}oltzmann equation
  without angular cutoff.
\newblock {\em Comm. Math. Phys.}, 304(2):513--581, 2011.

\bibitem{MR1055522}
Aleksei~A. Arsen{\cprime}ev and Olga~E. Buryak.
\newblock On a connection between the solution of the {B}oltzmann equation and
  the solution of the {L}andau-{F}okker-{P}lanck equation.
\newblock {\em Mat. Sb.}, 181(4):435--446, 1990.

\bibitem{bouchut}
Fran\c{c}ois Bouchut.
\newblock Hypoelliptic regularity in kinetic equations.
\newblock {\em J. Math. Pures Appl. (9)}, 81(11):1135--1159, 2002.

\bibitem{BGP}
Fran{\c{c}}ois Bouchut, Fran{\c{c}}ois Golse, and Mario Pulvirenti.
\newblock {\em Kinetic equations and asymptotic theory}, volume~4 of {\em
  Series in Applied Mathematics (Paris)}.
\newblock Gauthier-Villars, \'Editions Scientifiques et M\'edicales Elsevier,
  Paris, 2000.
\newblock Edited and with a foreword by Beno{\^{\i}}t Perthame and Laurent
  Desvillettes.

\bibitem{CV2}
Luis Caffarelli, Chi~Hin Chan, and Alexis Vasseur.
\newblock Regularity theory for parabolic nonlinear integral operators.
\newblock {\em J. Amer. Math. Soc.}, 24(3):849--869, 2011.

\bibitem{CV1}
Luis~A. Caffarelli and Alexis Vasseur.
\newblock Drift diffusion equations with fractional diffusion and the
  quasi-geostrophic equation.
\newblock {\em Ann. of Math. (2)}, 171(3):1903--1930, 2010.

\bibitem{CSS}
Stephen Cameron, Luis Silvestre, and Stanley Snelson.
\newblock Global a priori estimates for the inhomogeneous landau equation with
  moderately soft potentials.
\newblock {\em Preprint
  \href{https://arxiv.org/abs/1701.08215}{arXiv:1701.08215}}, 2017.

\bibitem{V2}
Maria~C. Caputo and Alexis Vasseur.
\newblock Global regularity of solutions to systems of reaction-diffusion with
  sub-quadratic growth in any dimension.
\newblock {\em Comm. Partial Differential Equations}, 34(10-12):1228--1250,
  2009.

\bibitem{MR2557895}
Hua Chen, Wei-Xi Li, and Chao-Jiang Xu.
\newblock Analytic smoothness effect of solutions for spatially homogeneous
  {L}andau equation.
\newblock {\em J. Differential Equations}, 248(1):77--94, 2010.

\bibitem{cdh}
Yemin Chen, Laurent Desvillettes, and Lingbing He.
\newblock Smoothing effects for classical solutions of the full {L}andau
  equation.
\newblock {\em Arch. Ration. Mech. Anal.}, 193(1):21--55, 2009.

\bibitem{MR2820356}
Yemin Chen and Lingbing He.
\newblock Smoothing estimates for {B}oltzmann equation with full-range
  interactions: spatially homogeneous case.
\newblock {\em Arch. Ration. Mech. Anal.}, 201(2):501--548, 2011.

\bibitem{MR2885564}
Yemin Chen and Lingbing He.
\newblock Smoothing estimates for {B}oltzmann equation with full-range
  interactions: {S}patially inhomogeneous case.
\newblock {\em Arch. Ration. Mech. Anal.}, 203(2):343--377, 2012.

\bibitem{DeG56}
Ennio De~Giorgi.
\newblock Sull'analiticit\`a delle estremali degli integrali multipli.
\newblock {\em Atti Accad. Naz. Lincei. Rend. Cl. Sci. Fis. Mat. Nat. (8)},
  20:438--441, 1956.

\bibitem{DeG}
Ennio De~Giorgi.
\newblock Sulla differenziabilit\`a e l'analiticit\`a delle estremali degli
  integrali multipli regolari.
\newblock {\em Mem. Accad. Sci. Torino. Cl. Sci. Fis. Mat. Nat. (3)}, 3:25--43,
  1957.

\bibitem{MR1324404}
Laurent Desvillettes.
\newblock About the regularizing properties of the non-cut-off {K}ac equation.
\newblock {\em Comm. Math. Phys.}, 168(2):417--440, 1995.

\bibitem{d04}
Laurent Desvillettes.
\newblock Plasma kinetic models: the {F}okker-{P}lanck-{L}andau equation.
\newblock In {\em Modeling and computational methods for kinetic equations},
  Model. Simul. Sci. Eng. Technol., pages 171--193. Birkh\"auser Boston,
  Boston, MA, 2004.

\bibitem{d15}
Laurent Desvillettes.
\newblock Entropy dissipation estimates for the {L}andau equation in the
  {C}oulomb case and applications.
\newblock {\em J. Funct. Anal.}, 269(5):1359--1403, 2015.

\bibitem{dvI}
Laurent Desvillettes and C{\'e}dric Villani.
\newblock On the spatially homogeneous {L}andau equation for hard potentials.
  {I}. {E}xistence, uniqueness and smoothness.
\newblock {\em Comm. Partial Differential Equations}, 25(1-2):179--259, 2000.

\bibitem{dv}
Laurent Desvillettes and C{\'e}dric Villani.
\newblock On the spatially homogeneous {L}andau equation for hard potentials.
  {II}. {$H$}-theorem and applications.
\newblock {\em Comm. Partial Differential Equations}, 25(1-2):261--298, 2000.

\bibitem{Di1}
Emmanuele DiBenedetto.
\newblock On the local behaviour of solutions of degenerate parabolic equations
  with measurable coefficients.
\newblock {\em Ann. Scuola Norm. Sup. Pisa Cl. Sci. (4)}, 13(3):487--535, 1986.

\bibitem{L3}
Emmanuele DiBenedetto, Ugo Gianazza, and Vincenzo Vespri.
\newblock Forward, backward and elliptic {H}arnack inequalities for
  non-negative solutions to certain singular parabolic partial differential
  equations.
\newblock {\em Ann. Sc. Norm. Super. Pisa Cl. Sci. (5)}, 9(2):385--422, 2010.

\bibitem{L2}
Emmanuele DiBenedetto, Ugo Gianazza, and Vincenzo Vespri.
\newblock Harnack type estimates and {H}\"older continuity for non-negative
  solutions to certain sub-critically singular parabolic partial differential
  equations.
\newblock {\em Manuscripta Math.}, 131(1-2):231--245, 2010.

\bibitem{L4}
Emmanuele DiBenedetto, Ugo Gianazza, and Vincenzo Vespri.
\newblock {\em Harnack's inequality for degenerate and singular parabolic
  equations}.
\newblock Springer Monographs in Mathematics. Springer, New York, 2012.

\bibitem{DPL}
Ron~J. DiPerna and Pierre-Louis Lions.
\newblock Global weak solutions of {V}lasov-{M}axwell systems.
\newblock {\em Comm. Pure Appl. Math.}, 42(6):729--757, 1989.

\bibitem{K1}
Matthieu Felsinger and Moritz Kassmann.
\newblock Local regularity for parabolic nonlocal operators.
\newblock {\em Comm. Partial Differential Equations}, 38(9):1539--1573, 2013.

\bibitem{MR0338568}
M.~Giaquinta and E.~Giusti.
\newblock Partial regularity for the solutions to nonlinear parabolic systems.
\newblock {\em Ann. Mat. Pura Appl. (4)}, 97:253--266, 1973.

\bibitem{gm}
M.~Giaquinta and G.~Modica.
\newblock Regularity results for some classes of higher order nonlinear
  elliptic systems.
\newblock {\em J. Reine Angew. Math.}, 311/312:145--169, 1979.

\bibitem{gs}
Mariano Giaquinta and Michael Struwe.
\newblock On the partial regularity of weak solutions of nonlinear parabolic
  systems.
\newblock {\em Math. Z.}, 179(4):437--451, 1982.

\bibitem{gv}
Fran\c{c}ois Golse and Alexis~F. Vasseur.
\newblock H\"older regularity for hypoelliptic kinetic equations with rough
  diffusion coefficients, 2015.
\newblock Version 2.

\bibitem{GLPS}
Fran{\c{c}}ois Golse, Pierre-Louis Lions, Beno{\^{\i}}t Perthame, and R{\'e}mi
  Sentis.
\newblock Regularity of the moments of the solution of a transport equation.
\newblock {\em J. Funct. Anal.}, 76(1):110--125, 1988.

\bibitem{gps}
Fran{\c{c}}ois Golse, Beno{\^{\i}}t Perthame, and R{\'e}mi Sentis.
\newblock Un r\'esultat de compacit\'e pour les \'equations de transport et
  application au calcul de la limite de la valeur propre principale d'un
  op\'erateur de transport.
\newblock {\em C. R. Acad. Sci. Paris S\'er. I Math.}, 301(7):341--344, 1985.

\bibitem{V3}
Thierry Goudon and Alexis Vasseur.
\newblock Regularity analysis for systems of reaction-diffusion equations.
\newblock {\em Ann. Sci. \'Ec. Norm. Sup\'er. (4)}, 43(1):117--142, 2010.

\bibitem{MR2914961}
Philip~T. Gressman, Joachim Krieger, and Robert~M. Strain.
\newblock A non-local inequality and global existence.
\newblock {\em Adv. Math.}, 230(2):642--648, 2012.

\bibitem{MR2784329}
Philip~T. Gressman and Robert~M. Strain.
\newblock Global classical solutions of the {B}oltzmann equation without
  angular cut-off.
\newblock {\em J. Amer. Math. Soc.}, 24(3):771--847, 2011.

\bibitem{MR3599518}
Maria Gualdani and Nestor Guillen.
\newblock Estimates for radial solutions of the homogeneous {L}andau equation
  with {C}oulomb potential.
\newblock {\em Anal. PDE}, 9(8):1772--1809, 2016.

\bibitem{MR1946444}
Yan Guo.
\newblock The {L}andau equation in a periodic box.
\newblock {\em Comm. Math. Phys.}, 231(3):391--434, 2002.

\bibitem{guo}
Yan Guo.
\newblock The {V}lasov-{P}oisson-{B}oltzmann system near {M}axwellians.
\newblock {\em Comm. Pure Appl. Math.}, 55(9):1104--1135, 2002.

\bibitem{hormander}
Lars H{\"o}rmander.
\newblock Hypoelliptic second order differential equations.
\newblock {\em Acta Math.}, 119:147--171, 1967.

\bibitem{im}
Cyril Imbert and Cl\'ement Mouhot.
\newblock H\"older continuity of solutions to hypoelliptic equations with
  bounded measurable coefficients.
\newblock Technical report, arXiv:1505.04608, 2015.
\newblock Version 5.

\bibitem{im-book}
Cyril Imbert and Luis Silvestre.
\newblock An introduction to fully nonlinear parabolic equations.
\newblock In {\em An introduction to the {K}\"ahler-{R}icci flow}, volume 2086
  of {\em Lecture Notes in Math.}, pages 7--88. Springer, Cham, 2013.

\bibitem{K}
Moritz Kassmann.
\newblock A priori estimates for integro-differential operators with measurable
  kernels.
\newblock {\em Calc. Var. Partial Differential Equations}, 34(1):1--21, 2009.

\bibitem{kolm}
Andrei Kolmogoroff.
\newblock Zuf\"allige {B}ewegungen (zur {T}heorie der {B}rownschen {B}ewegung).
\newblock {\em Ann. of Math. (2)}, 35(1):116--117, 1934.

\bibitem{MR2901061}
Joachim Krieger and Robert~M. Strain.
\newblock Global solutions to a non-local diffusion equation with quadratic
  non-linearity.
\newblock {\em Comm. Partial Differential Equations}, 37(4):647--689, 2012.

\bibitem{L1}
Olga~A. Ladyzhenskaya and Nina~N. Ural{\cprime}tseva.
\newblock {\em Linear and quasilinear elliptic equations}.
\newblock Translated from the Russian by Scripta Technica, Inc. Translation
  editor: Leon Ehrenpreis. Academic Press, New York-London, 1968.

\bibitem{PLLCam}
Pierre-Louis Lions.
\newblock On {B}oltzmann and {L}andau equations.
\newblock {\em Philos. Trans. Roy. Soc. London Ser. A}, 346(1679):191--204,
  1994.

\bibitem{MR3191417}
Shuangqian Liu and Xuan Ma.
\newblock Regularizing effects for the classical solutions to the {L}andau
  equation in the whole space.
\newblock {\em J. Math. Anal. Appl.}, 417(1):123--143, 2014.

\bibitem{moser}
J{\"u}rgen Moser.
\newblock A {H}arnack inequality for parabolic differential equations.
\newblock {\em Comm. Pure Appl. Math.}, 17:101--134, 1964.

\bibitem{nash}
John~F. Nash.
\newblock Continuity of solutions of parabolic and elliptic equations.
\newblock {\em Amer. J. Math.}, 80:931--954, 1958.

\bibitem{pp}
Andrea Pascucci and Sergio Polidoro.
\newblock The {M}oser's iterative method for a class of ultraparabolic
  equations.
\newblock {\em Commun. Contemp. Math.}, 6(3):395--417, 2004.

\bibitem{rs}
Linda~Preiss Rothschild and E.~M. Stein.
\newblock Hypoelliptic differential operators and nilpotent groups.
\newblock {\em Acta Math.}, 137(3-4):247--320, 1976.

\bibitem{luislandau}
Luis Silvestre.
\newblock Upper bounds for parabolic equations and the {L}andau equation.
\newblock Technical report, arxiv 1511.03248, 2016.
\newblock Version 2.

\bibitem{MR627929}
Michael Struwe.
\newblock On the {H}\"older continuity of bounded weak solutions of quasilinear
  parabolic systems.
\newblock {\em Manuscripta Math.}, 35(1-2):125--145, 1981.

\bibitem{V4}
Alexis Vasseur.
\newblock Higher derivatives estimate for the 3{D} {N}avier-{S}tokes equation.
\newblock {\em Ann. Inst. H. Poincar\'e Anal. Non Lin\'eaire},
  27(5):1189--1204, 2010.

\bibitem{vasseur}
Alexis~F. Vasseur.
\newblock The {D}e {G}iorgi method for elliptic and parabolic equations and
  some applications.
\newblock Preprint, to appear in Lectures on the Analysis of Nonlinear Partial
  Differential Equations Vol. 4.

\bibitem{molecules}
C\'edric Villani.
\newblock On the spatially homogeneous {L}andau equation for {M}axwellian
  molecules.
\newblock {\em Math. Models Methods Appl. Sci.}, 8(6):957--983, 1998.

\bibitem{wz09}
WenDong Wang and LiQun Zhang.
\newblock The {$C^\alpha$} regularity of a class of non-homogeneous
  ultraparabolic equations.
\newblock {\em Sci. China Ser. A}, 52(8):1589--1606, 2009.

\bibitem{wz11}
Wendong Wang and Liqun Zhang.
\newblock The {$C^\alpha$} regularity of weak solutions of ultraparabolic
  equations.
\newblock {\em Discrete Contin. Dyn. Syst.}, 29(3):1261--1275, 2011.

\bibitem{MR3158719}
Kung-Chien Wu.
\newblock Global in time estimates for the spatially homogeneous {L}andau
  equation with soft potentials.
\newblock {\em J. Funct. Anal.}, 266(5):3134--3155, 2014.

\end{thebibliography}

\signfg

\signci 

\signcm 

\signav

\end{document}